\documentclass{article}

\title{On periodicity of geodesic continued fractions}
\author{Hohto Bekki
\footnote{
bekki@ms.u-tokyo.ac.jp}}
\date{
\small{\it Graduate School of Mathematical Sciences, The University of Tokyo, 3-8-1 Komaba, Meguro, Tokyo, 153-8914 Japan}}

\usepackage{geometry}
\usepackage{amscd}                		
\usepackage{graphicx}			
\usepackage{amssymb}
\usepackage{amsmath}
\usepackage{amsthm}
\usepackage{amsfonts}
\usepackage{pb-diagram}
\usepackage{array}
\usepackage{multirow}
\usepackage{dcolumn}
\usepackage{eucal}
\usepackage{enumerate}
\usepackage{comment}
\usepackage{multicol}

\theoremstyle{plain}
\newtheorem{thm}{Theorem}[subsection]
\newtheorem{lem}[thm]{Lemma}
\newtheorem{claim}{Claim}
\newtheorem{prop}[thm]{Proposition}
\newtheorem{cor}[thm]{Corollary}
\newtheorem{theo}{Theorem}

\theoremstyle{definition}
\newtheorem{dfn}[thm]{Definition}

\newtheorem{ex}{Example}
\newtheorem{algo}{Algorithm}

\theoremstyle{remark}
\newtheorem{rmk}{Remark}

\newcommand{\Gal}{\mathop{\mathrm{Gal}}\nolimits}

\newcommand{\End}{\mathop{\mathrm{End}}\nolimits}

\def\hsymb#1{\mbox{\strut\rlap{\smash{\Large$#1$}}\quad}}

\begin{document}
\maketitle

\begin{abstract}
In this paper, we study geodesic generalizations of Lagrange's theorem in the classical theory of continued fractions. We first construct a geometric framework which generalizes the geodesic interpretation of Lagrange's theorem by Sarnak. Then we develop a theory of geodesic continued fractions originally studied by Lagarias and Beukers into a form which is applicable to our geometric framework. As a result, we establish algorithms to expand bases of number fields with rank one unit groups, and prove their periodicity. Furthermore, we show that their periods describe the unit groups of the number fields. Our arguments are ad\'elic so that we can also treat $p$-units. In this case, our algorithm is an refinement of Mahler's $p$-adic continued fractions. We also treat relative quadratic extensions of number fields with rank one relative unit groups.
~
\end{abstract}

\tableofcontents
\section{Introduction}
	The classical continued fraction expansion of a real number $\alpha$ expresses $\alpha$ in the following form,
\begin{eqnarray*}
\alpha &=& a_0 + \cfrac{1}{a_1+ \cfrac{1}{\ddots+\cfrac{1}{a_n}}} = [a_0,a_1,\dots, a_n] , \ {\rm or} \\
  \alpha &=& a_0 + \cfrac{1}{a_1+ \cfrac{1}{a_2+\cfrac{1}{\ddots}}} = [a_0,a_1, a_2, \dots] \ \ (a_0 \in \mathbb{Z}, \ a_k \in \mathbb{N}_{>0}\ \text{for }k \geq 1).
\end{eqnarray*}
Then, $\dfrac{p_k}{q_k}= [a_0,a_1,\dots, a_k] $ for relatively prime integers $p_k \in \mathbb{Z}, q_k \in \mathbb{N}_{>0}$ ($k \in \mathbb{N}$) is called the $k$-th convergent of $\alpha$.
 
 Our motivation for this paper is to study generalizations of the following classical theorem.
\begin{theo}[Lagrange's theorem]\label{lag}~
\begin{enumerate}[{\rm (i)}]
\item The continued fraction expansion of a real number $\alpha$ is periodic if and only if $\alpha$ is a real quadratic irrational.
\item Let $\alpha$ be a real quadratic irrational and suppose its continued fraction expansion $\alpha =  [a_0,a_1, a_2, \dots] $ is purely periodic. That is, there exists $l \in \mathbb{N}_{>0}$ such that $a_{k+l} = a_k$ for all $k \in \mathbb{N}$. Let $l$ be the minimal among these. Let $\mathcal{O} = \{x \in \mathbb{Q}(\alpha)\ |\ x(\mathbb{Z}+\mathbb{Z}\alpha) \subset \mathbb{Z}+\mathbb{Z}\alpha\}$, the order associated to $\alpha$. Then $u = p_{l-2}\alpha + p_{l-1}$ is a fundamental unit of $\mathcal{O}$, i.e. a generator of $\mathcal{O}^{\times}$ modulo $\pm 1$, where, for $k \in \mathbb{N}$, $p_k/q_k$ is the $k$-th convergent of $\alpha$.
\end{enumerate}
\end{theo}
 
Our starting point is the following geometric interpretation of Lagrange's theorem showed by Sarnak~\cite{sarnak}.

\begin{theo}[Geodesic Lagrange's theorem  \cite{sarnak}]\label{sarnak-lag}~
\begin{enumerate}[{\rm (i)}]
\item A geodesic in the Poincar\'e upper-half plane $\mathfrak{h}$ connecting two real numbers $\alpha , \beta \in \mathbb{R} = \partial \mathfrak{h}$ is periodic with respect to the natural action of $SL_2(\mathbb{Z})$ if and only if $\alpha$ and $\beta$ are conjugate real quadratic irrationals. Here ``periodic'' means that the image of the geodesic under the natural projection to the modular curve $SL_2(\mathbb{Z})\backslash \mathfrak{h}$ defines a closed geodesic.
\item If this is the case, let $l \in \mathbb{R}_{>0}$ be the length of this closed geodesic. Let $\mathcal{O} = \{x \in \mathbb{Q}(\alpha)\ |\ x(\mathbb{Z}\alpha +\mathbb{Z}\beta) \subset \mathbb{Z}\alpha +\mathbb{Z}\beta \}$, the order associated to $(\alpha, \beta)$. Then $u=\exp (l/2)$ is a fundamental unit of $\mathcal{O}$.
\end{enumerate}
\end{theo}

Our plan is as follows. First, in Section \ref{heeg}, we present a geometric framework in which a generalization of Theorem \ref{sarnak-lag} can be considered and proved. In Section \ref{defheeg}, we construct a map 
$$\varpi_S : T_S^1 \rightarrow \mathfrak{h}^n_S$$ 
associated to a $\mathbb{Q}$-basis of a number field of degree $n$ and a certain finite set $S$ of places of $\mathbb{Q}$ containing the archimedean place $\infty$, where $T_S^1$ is a parameter space and $\mathfrak{h}^n_S$ is a generalized upper-half space $GL_n(\mathbb{Q}_S)/\mathbb{Q}_S^{\times}K$. Here, $\mathbb{Q}_S:= \prod_{v \in S}\mathbb{Q}_v$ is the ring of $S$-ad\`eles of $\mathbb{Q}$, and $K$ is a certain maximal compact subgroup of $GL_n(\mathbb{Q}_S)$. In the simplest case where $n=2$ and $S=\{\infty\}$, $\mathfrak{h}^n_S$ is the usual Poincar\'e upper half plane $\mathfrak{h}$, and the map $\varpi_S$ composed with the natural projection $\pi : \mathfrak{h} \rightarrow SL_2(\mathbb{Z}) \backslash \mathfrak{h}$ defines either a Heegner point or a closed geodesic depending on the number field is imaginary quadratic or real quadratic, respectively. Darmon~\cite{darmon} uses the term Heegner objects to indicate such Heegner points and closed geodesics simultaneously, and we call these $\varpi_S$ also Heegner objects. We are interested in the periods of the map 
$$\pi \circ \varpi_S :T_S^1 \rightarrow \Gamma_S \backslash \mathfrak{h}_S^n,$$
where $\Gamma_S = SL_n(\mathbb{Z}[1/N]), N:= \prod_{p \in S \text{:fin}}p$, is a discrete group acting on $\mathfrak{h}_S^n$, and $\pi: \mathfrak{h}_S^n \rightarrow \Gamma_S \backslash \mathfrak{h}_S^n$ is the natural projection. We define the ``period group'' $\Gamma_{\varpi_S} \subset \Gamma_S$ which captures the periods of this map. In Sections \ref{sectorderunits} and \ref{sectper}, we prove the following generalization of Theorem \ref{sarnak-lag}.
\begin{theo}[Proposition \ref{orderunits}, Corollary \ref{corunitperiod}, Theorem \ref{perHeeg}]\label{theoheeg}~
\begin{enumerate}[{\rm (i)}]
\item The image of $\pi \circ \varpi_S$ in $\Gamma_S \backslash \mathfrak{h}_S^n$ is compact.
\item The period group $\Gamma_{\varpi_S}$ is isomorphic to the norm one unit group $\mathcal{O}_{w,S}^1$ in a natural way. 
(Here, $\mathcal{O}_{w,S}^1$ is the group of norm one units in $\mathcal{O}_{w,S}$, a certain $\mathbb{Z}[1/N]$-order in $F$ associated to $w$. See Definition \ref{dfnorderunits}.)
\end{enumerate}
 \end{theo}
 In Section \ref{sectchi}, we extend Theorem \ref{theoheeg} to the $\chi$-components of unit groups for ``quadratic characters'' $\chi$.\\

  Then, we restrict ourselves to the case where the ``rank'' of $\varpi_S$ is one. In this case, we give algorithms to calculate the period group $\Gamma_{\varpi_S}$, and consider generalizations of Theorem \ref{lag}.
 
 Section \ref{gcf} is devoted to the case where $S=\{\infty\}$. In this case, $\mathfrak{h}_S^n$ has a natural structure of a Riemannian manifold and a one dimensional Heegner object $\varpi_S$ is a geodesic on $\mathfrak{h}_S^n$. In Sections \ref{gcf1} and \ref{siegelsets}, we present an algorithm (Algorithm \ref{GCF}) which is a slight generalization of the geodesic continued fractions originally studied by Lagarias~\cite{lagarias} and Beukers~\cite{beukers} in the area of Diophantine approximations. The algorithm ``expands'' a certain class of geodesics on $\mathfrak{h}_S^n$ including all Heegner objects, and provides a sequence $\{(A_k, B_k, \varpi_k)\}_{k \in \mathbb{Z}}$, where $A_k, B_k \in \Gamma_S$, and $\varpi_k=B_k^{-1}\varpi$ for such a geodesic $\varpi$. 
 In Section \ref{sectapp}, we prove the following generalization of Theorem \ref{lag}.
\begin{theo}[Theorem \ref{arch-main}]\label{theoarch-main}
Let $\varpi_S$ be a one dimensional Heegner object. Let $\{(A_k, B_k, \varpi_{S, k})\}_{k \in \mathbb{Z}}$ be the sequence obtained by Algorithm \ref{GCF} applied to $\varpi_S$.
\begin{enumerate}[{\rm (i)}]
\item Then there exist $k_0,  k_1 \in \mathbb{N}, k_0<k_1$, and $\rho \in T_S^1, \rho \neq 1$, such that $\varpi_{S,k_0}(t) = \varpi_{S,k_1}(\rho t)$ for all $t \in T_S^1$, that is, the algorithm is ``periodic''.
\item Moreover, $B_{k_1}B_{k_0}^{-1} \in \Gamma_{\varpi_S}$ generates a finite index subgroup of $\Gamma_{\varpi_S} \simeq \mathcal{O}_{w,S}^1$.
\end{enumerate}
\end{theo} 
We also prove the $\chi$-component version of Theorem \ref{theoarch-main} in Section \ref{sectapp}. 

Section \ref{p-gcf} is devoted to the case where $S=\{\infty, p\}, n=2$ for a prime number $p$. In this case a Heegner object $\varpi_S$ is a $p$-geodesic, which we define in Definition \ref{p-geo}. In Sections \ref{sectp-FD} and \ref{sectp-gcf}, we present a $p$-geodesic version of Algorithm \ref{GCF} (Algorithm \ref{p-gcfalgo}) which ``expands'' a $p$-geodesic $\varpi$ and provides a sequence $\{(A_k, B_k, \varpi_k)\}_{k \in \mathbb{Z}}$, where $A_k, B_k \in \Gamma_S$, and $\varpi_k=B_k^{-1}\varpi$. 
We should remark that Algorithm \ref{p-gcfalgo} is essentially the same algorithm as in Mahler~\cite{mahler}, and its periodicity has already been studied by Mahler~\cite{mahler}, Deanin~\cite{deanin} and de Weger~\cite{deweger}. 
In Section \ref{sectp-cf}, we develop Algorithm \ref{p-gcfalgo} into a more intuitive form which is similar to the classical continued fractions of real numbers. We define $\{\infty,p\}$-continued fraction expansion (Algorithm \ref{p-cf}) which formally expands a number $z \in \mathfrak{h}\cap \mathbb{Q}_p \subset \mathbb{C} \simeq \overline{\mathbb{Q}}_p$ (we fix an isomorphism $\mathbb{C} \simeq \overline{\mathbb{Q}}_p$ of fields) in the following form. 
\begin{eqnarray*}
z &=&\delta _1( a_1 +p^2b_{10}- \cfrac{p^2}{b_{11}- \cfrac{1}{{\ddots}-\cfrac{1}{b_{1l_1}-\cfrac{1}{\delta _2(a_2+p^2b_{20}-\cfrac{p^2}{b_{21}-\cfrac{1}{\ddots}^{\ }})^{\delta _2 }}^{\ }}^{\ }}^{\ }}^{\ })^{\delta _1}\\
   &=& [\  \delta _1 ;\  a_1 ;\  b_{10}, \dots , b_{1l_1}  ;\   \delta _2 ; \ a_2 ;\   b_{20}, \dots , b_{2l_2}  ;  \dots ].
\end{eqnarray*}
We prove the following generalization of Theorem \ref{lag}.
\begin{theo}[Theorem \ref{p-lag}, Theorem \ref{p-lag2}]\label{theop-main}~
\begin{enumerate}[{\rm (i)}]
\item If $z \in \mathfrak{h} \cap \mathbb{Q}_p \subset \mathbb{C} = \overline{\mathbb{Q}}_p$ is an imaginary quadratic irrational, then its $\{\infty,p\}$-continued fraction expansion is periodic. 
\item Let $F$ be an imaginary quadratic field $\mathbb{Q}(\sqrt{-d}) \subset \mathbb{C} = \overline{\mathbb{Q}}_p$, where $d$ is a positive square free integer. Suppose $F \subset \mathbb{Q}_p$. Let $\theta = \dfrac{-1+\sqrt{-d}}{2}$ if $-d \equiv 1\mod 4$, and $\theta = \sqrt{-d}$ otherwise. Then the $\{\infty,p\}$-continued fraction expansion of $\theta$ is purely periodic and its period provides a fundamental norm one $p$-unit of $F$. (See Section \ref{sectp-app} for the details.)
\end{enumerate}
\end{theo}
As an application of Theorem \ref{theop-main}, we present a method to find solutions to the Pell-like equations, which is an analog of the method to find solutions to the Pell equations using the classical continued fractions.

In short, we obtain generalizations of Lagrange's theorem in the following cases.
\begin{list}{}{}
\item[(a1)] Real quadratic, complex cubic, and totally imaginary quartic fields and its unit groups.
\item[(a2)] Quadratic extensions over totally real fields with exactly two real places and its relative unit groups.
\item[(a3)] Totally imaginary quadratic extensions over ATR fields, i.e. fields with exactly one complex place, and its relative unit groups.
\item[(b)] Imaginary quadratic fields and its norm one $p$-unit groups for a prime number $p$ which splits completely.
\end{list}
Note that (a1)$\sim$(a3) can be summarized as follows.
\begin{list}{}{}
\item[(a)] Extensions of number fields with rank one relative unit groups and its relative unit groups.
\end{list}
In all of these cases, the period of our algorithm gives a non-torsion unit of the number field. Moreover, in (b), the period of our algorithm gives a fundamental norm one $p$-unit, i.e. a generator of the norm one $p$-unit group modulo torsion subgroup.

In Section \ref{sectex}, we present some numerical examples of Algorithms \ref{GCF} and \ref{p-cf} and their periodicity.

\paragraph{Relation to preceding works} 
We briefly remark what is new and what is not in this paper.
\subparagraph{Heegner objects}
We can find essentially the same objects as our Heegner objects in some preceding works in slightly different directions, e.g. Hiroe, Oda~\cite{hiroeoda}, Oh~\cite{oh} and Einsiedler, et.al.~\cite{einsiedler}. The periodicity (i.e. the compactness) of these objects are discussed in all of these studies. Therefore, (i) of Theorem \ref{theoheeg}, more precisely, Proposition \ref{orderunits} and Theorem \ref{perHeeg}, are not especially new results.
 
 On the other hand, it seems no attempt has been made to apply these objects and its periodicity to the theory of multidimensional continued fractions. We do this by introducing the ``period group'' $\Gamma_{\varpi_S}$ which connects the unit groups $\mathcal{O}_{w,S}^1$ to the periods of our geodesic continued fractions. The theory on the $\chi$-components is also new in this paper.
 
\subparagraph{Geodesic continued fractions}
The theory of geodesic continued fractions was originally studied by Lagarias~\cite{lagarias} and Beukers~\cite{beukers} to find good Diophantine approximations. However, as Beukers~\cite[p.642]{beukers} remarked, no analogue of  Lagrange's theorem was known for geodesic continued fractions. In this paper, we establish a generalization of Lagrange's theorem by choosing geodesics, to be Heegner objects, different from those studied in \cite{lagarias} and \cite{beukers}. In order to do this, we extend the class of geodesics to which the algorithm is applicable by introducing the new notion of weak convexity (Definition \ref{FD}) instead of the convexity in the previous works.
 
\subparagraph{$\{\infty,p\}$-continued fractions}
 As we remarked above, the Algorithm \ref{p-gcfalgo} and its periodicity is studied in Mahler~\cite{mahler}, Deanin~\cite{deanin}, and de Weger~\cite{deweger}. 
 On the other hand, as far as the author knows, Algorithm \ref{p-cf}, which can be seen as a refinement of Algorithm \ref{p-gcfalgo}, and the results that follows are new.

\section{Preliminaries on Heegner objects}\label{heeg} 
 In this section we construct a totally geodesic submanifolds in the generalized upper half space associated to bases of number fields. For $n=2$, these are the Heegner points and closed geodesics on the modular curve $SL_2(\mathbb{Z}) \backslash \mathfrak{h}$ associated to imaginary and real quadratic fields, respectively. We call these objects Heegner objects following Darmon~\cite{darmon}. 

\subsection{Definitions}\label{defheeg}
 For a positive integer $n$ and a prime number $p$, let $\mathfrak{h}^n =\mathfrak{h}^n_{\infty} := GL_n(\mathbb{R})/\mathbb{R}^{\times}O(n)$ and let $\mathfrak{h}^n_p := GL_n(\mathbb{Q}_p)/\mathbb{Q}_p^{\times} GL_n(\mathbb{Z}_p)$. See Goldfeld~\cite{goldfeld}, \cite{goldfeld2}, Terras~\cite{terras}, \cite{terras2}, or Borel~\cite{borel} for basic references.  For $S$, a finite set of places of $\mathbb{Q}$, let $\mathfrak{h}^n_S$ be the generalized $S$-upper half space $\prod_{v \in S} \mathfrak{h}^n_v$. Then $SL_n(\mathbb{Q}_S)$ $(\mathbb{Q}_S:= \prod_{v \in S} \mathbb{Q}_v)$ acts on $\mathfrak{h}^n_S$ from the left. For $A=(a_{ij}) \in GL_n(\mathbb{R})$ (resp. $GL_n(\mathbb{Q}_p)$), we denote by $[A]$ or $[a_{ij}]$ the class of $A$ in $\mathfrak{h}_{\infty}^n=GL_n(\mathbb{R})/\mathbb{R}^{\times}O(n)$ (resp. $\mathfrak{h}_p^n=GL_n(\mathbb{Q}_p)/\mathbb{Q}_p^{\times}GL_n(\mathbb{Z}_p)$).
 
 In the following we fix an isomorphism $\mathbb{C} \simeq \mathbb{R}^2 ; x+iy \mapsto (y,x)$ of $\mathbb{R}$-vector spaces. 
 
 Let $F/ \mathbb{Q}$ be an extension of degree $n$, and let $S$ be a finite set of places of $\mathbb{Q}$ consisting of the archimedean place $\infty$ and finite places which is unramified in $F/ \mathbb{Q}$. We denote by $S^{\infty}$ the set of finite places in $S$, and by $S_F$ (resp. $S_{F,p}$) the set of places of $F$ above $S$ (resp. $p \in S$).
 
 Let $w _1 \cdots w _n \in F$ be a basis of $F$ over $\mathbb{Q}$. Set $w := {}^t\!(w _1 \cdots w _n) \in F^n$. Here, for a matrix $A$, we denote by ${}^tA$ the transpose matrix of $A$.
 
 For the archimedean part, we denote by $\sigma _1, \cdots, \sigma _{r_1}$ the real ebmeddings, and by $\tau _1 = \sigma _{r_1+1}, \cdots, \tau _{r_2} = \sigma _r, \bar{\tau} _1 = \sigma _{r+1}, \cdots, \bar{\tau}_{r_2} = \sigma _n$ the complex embeddings. We simply write $\alpha^{(i)}:= \sigma _i(\alpha)$. We identify $\sigma_i ~ (1 \leq i \leq r)$ with the archimedean places of $F$, and denote by $F_{\sigma_i}$ the completion of $F$ with respect to $\sigma_i$, i.e. $F_{\sigma_i}=\mathbb{R}$ (resp. $\mathbb{C}$) if $\sigma_i$ is real (resp. complex).
 Set $n_i =n_{\sigma_i} = [F_{\sigma_i}:\mathbb{R}]$. Then we have the natural embedding
 $$\sigma = \sigma _1 \times \cdots \times \sigma _r : F \hookrightarrow F_{\infty} = \mathbb{R}^{r_1} \times \mathbb{C}^{r_2} \simeq \mathbb{R}^n ; 
 \alpha \mapsto (\alpha^{(1)}, \cdots, \alpha^{(r)}),$$
  where the last $\mathbb{R}^{r_1} \times \mathbb{C}^{r_2} \simeq \mathbb{R}^n$ is the isomorphism of $\mathbb{R}$-vector spaces defined by the fixed identification $\mathbb{C} \simeq \mathbb{R}^2$.
  
  We define the archimedean part of the Heegner object associated to the basis $w= {}^t\! (w_1, \dots, w_n)$ of $F$ over $\mathbb{Q}$ as
  $$\varpi =\varpi _\infty = \varpi _{\infty, w} : (\prod_{v|\infty}\mathbb{R}_{>0})^1 \rightarrow \mathfrak{h}^n ; (t_1, \cdots, t_r) \mapsto [t_1w^{(1)} \cdots t_rw^{(r)}].$$
 Here $(\prod_{v|\infty}\mathbb{R}_{>0})^1$ is the subgroup of $\prod_{v|\infty}\mathbb{R}_{>0}=\mathbb{R}_{>0}^{r}$ consisting of $(t_1, \dots, t_r) \in \mathbb{R}_{>0}^r$ such that $\prod_{i=1}^r t_i^{n_i} =1$. We regard $(t_1w^{(1)} \cdots t_rw^{(r)})$ as an element of $GL_n(\mathbb{R})$ under the identification $\mathbb{C} \simeq \mathbb{R}^2$, and $[t_1w^{(1)} \cdots t_rw^{(r)}]$ is the class of $(t_1w^{(1)} \cdots t_rw^{(r)})$ in $\mathfrak{h}^n$. Note that the regularity of the matrix holds since $\sigma(w)={}^t\!(\sigma(w_1), \dots, \sigma(w_n))$ is a basis of $F_{\infty}$ over $\mathbb{R}$.

 Similarly, for a finite place $p \in S^{\infty}$, let $v_1, \cdots, v_g$ be the places of $F$ above $p$. For each $v = v_i | p$, we denote by $F_v$ the completion of $F$ with respect to $v$, and by $\mathcal{O}_v$ the ring of integers. We fix an isomorphism $\mathcal{O}_v \simeq \mathbb{Z}_p^{n_v}$ of free $\mathbb{Z}_p$-modules, where $n_v =[F_v:\mathbb{Q}_p]$, and extend it to the isomorphism $F_v \simeq \mathbb{Q}_p^{n_v}$ of $\mathbb{Q}_p$-vector spaces.
 Then we have the embedding
 $$v_1 \times \cdots \times v_g : F \hookrightarrow F_p := \prod F_{v_i} \simeq \prod \mathbb{Q}_p^{n_{v_i}} = \mathbb{Q}_p^n,$$
 where the middle isomorphism of $\mathbb{Q}_p$-vector spaces is defined by the fixed isomorphisms.
 
  We define the $p$-part of the Heegner object with an additional index $e = (e_1, \dots , e_g) \in \mathbb{Z}^g$ as
$$ \varpi _p= \varpi _{p, w,e} : (\prod_{v | p} p^{\mathbb{Z}})^1 \rightarrow \mathfrak{h}^{n}_{p} ; (t_1, \cdots, t_g) \mapsto [p^{e_1}t_1^{-1}v_1(w) \cdots p^{e_g}t_g^{-1}v_g(w)],$$
 where $ (\prod_{v | p} p^{\mathbb{Z}})^1$ is the subgroup of $ \prod_{v | p} p^{\mathbb{Z}}$ consisting of $(t_1, \dots, t_g) \in \prod_{v | p} p^{\mathbb{Z}}$ such that $\prod_{i}t_i^{n_{v_i}}=1$, and $[p^{e_1}t_1^{-1}v_1(w) \cdots p^{e_g}t_g^{-1}v_g(w)]$ is the class of $(p^{e_1}t_1^{-1}v_1(w) \cdots p^{e_g}t_g^{-1}v_g(w))$ in $\mathfrak{h}^n_p$. Here again we used the fixed isomorphism $F_v \simeq \mathbb{Q}_p^{n_{v}}$ and regard $(p^{e_1}t_1^{-1}v_1(w) \cdots p^{e_g}t_g^{-1}v_g(w))$ as an element of $GL_n(\mathbb{Q}_p)$.

\begin{dfn}\label{dfnheeg} 
 Let the notations be as above. We define the $S$-Heegner object associated to the basis $w={}^t\!(w_1, \cdots, w_n)$ (with indices $(e_p)_
 {p\in S^{\infty}}$) as the product of local ones,
$$\varpi _S:=\prod_{p\in S} \varpi _{p,w,e_p}  : T_S^1 \rightarrow \mathfrak{h}_S^n = \prod_{p \in S} \mathfrak{h}^n_p,$$
where $T_S^1:= (\prod_{v|\infty}\mathbb{R}_{>0})^1 \times \prod_{\substack{p \in S^{\infty}}} (\prod_{v | p} p^{\mathbb{Z}})^1 $. 
\end{dfn}

\begin{rmk}
 The definition of $S$-Heegner object depends only on the basis $w$ (and $(e_p)_p$), and does not depend on the choice of order of places $\sigma _i | \infty$, $v_i |p$ and the isomorphisms $\mathcal{O}_v \simeq \mathbb{Z}_p^{n_v}$, since a change of the order of places is canceled by the right action of $O(n)$ and $GL_n(\mathbb{Z}_p)$, and replacement of the isomorphism $\mathcal{O}_v \simeq \mathbb{Z}_p^{n_v}$ is also canceled by the right action of $GL_n(\mathbb{Z}_p)$.
\end{rmk}

\begin{prop}\label{heeginj}
The Heegner object $\varpi_S: T_S^1 \rightarrow \mathfrak{h}_S^n$ is an injective map.
\end{prop}
\begin{proof}
It suffices to show $\varpi_{\infty}$ and $\varpi_p$ (for $p \in S^{\infty}$) are all injective. Set $W := (w^{(1)} \cdots w^{(r)}) \in GL_n(\mathbb{R})$ as in the definition of $\varpi_{\infty}$, so that $\varpi (t) =\varpi_{\infty} (t_1,\dots, t_r) = [ W
I(t)
]$. Here $I(t)=I(t_1,\dots, t_r) :=
 \left(
		\begin{array}{ccc}
		t _1I_{n_1}&& \\
		&\ddots& \\
		&&t_rI_{n_r}
		\end{array}
\right)$, and $I_{k}\ (k \in \mathbb{N}_{>0})$ is the $k \times k$ identity matrix. Suppose $\varpi (t)=\varpi (t')$ for $t,t' \in (\prod_{v|\infty}\mathbb{R}_{>0})^1$. Then there exist $\lambda \in \mathbb{R}^{\times}$ and $R \in O(n)$ such that $WI(t) = WI(t')\lambda R$. Thus $\lambda I(t'/t) \in O(n)$, where $t'/t =(t_1'/t_1,\dots, t_r'/t_r)$. This is only possible if $t=t'$. This proves that $\varpi_{\infty}$ is injective. The case for a finite place $p \in S^{\infty}$ is essentially the same, replacing $O(n)$ with $GL_n(\mathbb{Z}_p)$, and we omit the  detail.
\end{proof}

\subsection{Order and unit group associated to $w$}\label{sectorderunits}
 We keep the notations in Section \ref{defheeg}.
 Set $N:= \prod_{p\in S^{\infty}} p$, and $\Gamma _S := SL_n(\mathbb{Z}[1/N])$.
\begin{dfn}\label{dfnorderunits}
Let $w={}^t\!(w_1, \dots ,w_n)$ be a basis of $F$ over $\mathbb{Q}$ as above. Set, 
 $$\mathcal{A}_{w,S} =\mathcal{A}_{w,N} := \{ A \in M_n(\mathbb{Z}[1/N])\  |\ \exists \lambda \in F, \ Aw = \lambda w \} ,$$
 $$\ \ \Gamma_{w,S} =\Gamma _{w,N}:= \{ A \in SL_n(\mathbb{Z}[1/N])\  |\ \exists \lambda \in F, \ Aw = \lambda w \} .$$
Then we define 
$\varphi = \varphi_{w} : \mathcal{A}_{w,S} \rightarrow F $ 
by $Aw = \varphi (A) w$, and set $\mathcal{O}_{w,S} := \varphi(\mathcal{A}_{w,S}) \subset F$. That is, $\varphi$ is the partial inverse map of 
$$\iota_w:  F \hookrightarrow \End_{\mathbb{Q}}(F) \simeq M_n(\mathbb{Q}); a \mapsto m_a,$$
where $m_a \in \End_{\mathbb{Q}}(F) \simeq M_n(\mathbb{Q})$ is defined by the ``multiplication by $a$'': $x \mapsto ax$, and the latter isomorphism of rings is defined by taking $w$ as a basis of $F$ over $\mathbb{Q}$. Indeed, $\mathcal{A}_{w,S} = \iota_w(F) \cap M_n(\mathbb{Z}[1/N])$, and $\iota_w \circ \varphi (A)= A$ for $A \in \mathcal{A}_{w,S}$.

If $S=\{\infty\}$, we omit $S$ and simply write $\mathcal{A}_{w}, \Gamma_{w}$, etc. Similarly, if $S=\{p\}$, we simply write $\mathcal{A}_{w,p}, \Gamma_{w,p}$, etc.
\end{dfn}
 
 In the following proposition we relate $\mathcal{A}_{w,S}$ and $\Gamma_{w,S}$ to the number field $F$ through $\varphi$.
 
 \begin{prop}\label{orderunits} 
 \begin{enumerate}[{\rm (i)}]
  \item The map $\varphi$ is an injective ring homomorphism. The image $\mathcal{O}_{w,S}$ of $\mathcal{A}_{w,S}$ under $\varphi$ is an $\mathbb{Z}[1/N]$-order in $F$. Moreover, let $\mathfrak{a}:= \bigoplus_{i=1}^n \mathbb{Z}[1/N]w_i \subset F$. Then we have $\mathcal{O}_{w,S}=\{x \in F\ |\ x\mathfrak{a} \subset \mathfrak{a} \}$.
 \item The following diagram commutes:
 $$
\begin{tabular}{ccc}
$F$&$\overset{N_{F/\mathbb{Q}}}{\longrightarrow}$&$\mathbb{Q}$ \\
{\small $\varphi$}$ \uparrow$&&\rotatebox{90}{$\subset$} \\
$\mathcal{A}_{w,S}$&$\overset{\det}{\longrightarrow}$&$\mathbb{Z}[1/N].$\\
\end{tabular}
$$

\item We have an exact sequence, 
$$1 \rightarrow \Gamma_{w,S} \hookrightarrow \mathcal{A}_{w,S}^{\times} \overset{\det }{\rightarrow} \mathbb{Z}[1/N]^{\times}. $$
In particular, $\varphi$ induces an isomorphism of groups $\Gamma_{w,S} \overset{\sim}{\rightarrow} \mathcal{O}_{w,S}^1$, where $$\mathcal{O}_{w,S}^1:=\{x \in \mathcal{O}_{w,S}^{\times}\ |\  N_{F/\mathbb{Q}}(x)=1\}.$$
\item $$rank_{\mathbb{Z}} \Gamma_{w,S} = rank_{\mathbb{Z}} \mathcal{O}_{w,S}^{\times} - (\# S -1)= \sum_{p \in S} (g_p-1).$$
where, for p $\in S$, $g_p$ denotes the number of places of $F$ above $p$. $($If $p=\infty$, $g_{\infty}=r.)$
 \end{enumerate}
 \end{prop}
 
\begin{proof}
(i) Since $\iota_w$ is a homomorphism of rings, It is clear that $\varphi$ is an injective ring homomorphism from the identity $\iota_w \circ \varphi (A)= A$ for $A \in \mathcal{A}_{w,S}$.

 Since $\mathcal{A}_{w,S}$ is a $\mathbb{Z}[1/N]$-submodule of $M_n(\mathbb{Z}[1/N])$, it is finite over $\mathbb{Z}[1/N]$ as a module. To see that $\mathcal{O}_{w,S}$ is an order, it is enough to show that $\mathcal{O}_{w,S}$ contains a generator of $F$ over $\mathbb{Q}$. Consider the equation for $X \in GL_n(\mathbb{Q})$,
$$
X
\left(
\begin{array}{c}
w _1 \\
\vdots \\
w _n 
\end{array}
\right)
= w _i 
\left(
\begin{array}{c}
w _1 \\
\vdots \\
w _n 
\end{array}
\right)
.$$
Since $(w _1, \cdots, w _n)$ is a bisis of $F$ over $\mathbb{Q}$, this always has a solution. Then by multiplying the common denominator $d_i \in \mathbb{N}_{>0}$ of elements of $X$, we obtain $A=d_iX \in \mathcal{A}_{w,S}$ such that $\varphi (A) = d_i w _i $. 
We get the last assertion by applying $\varphi$ to the identity $\mathcal{A}_{w,S} = \iota_w(F) \cap M_n(\mathbb{Z}[1/N])$. This proves {\rm (i)}.

(ii) This follows from the identity $\iota_w \circ \varphi (A)= A$ for $A \in \mathcal{A}_{w,S}$ since $N_{F/\mathbb{Q}}(a) := \det (m_a)$.

(iii) This follows from {\rm (ii)}.

(iv) Since $pI \in \mathcal{A}_{w,S}^{\times}$ ($p|N$, $I$: identity matrix) maps to $p^n \in \mathbb{Z}[1/N]^{\times}$, the determinant map in {\rm (iii)} has a finite cokernel. Then {\rm (iv)} follows from {\rm (iii)} and Dirichlet's unit theorem for orders in number fields.
\end{proof} 

In the following we describe the arithmetic object $\Gamma_{w,S}$ through the geometric periods of the Heegner object $\varpi_S$ associated to $w$.
 
\subsection{Periodicity}\label{sectper}
 Let $\varpi_S$ be the ($S$-)Heegner object associated to a basis $w$ of $F$ over $\mathbb{Q}$ (with any index $(e_p)_{p \in S^{\infty}}$) as in Section \ref{defheeg}. We keep the notations above.  
 Let $\pi : \mathfrak{h} _S^n \rightarrow \Gamma _S \backslash \mathfrak{h} _S^n $ be the natural projection. 
  In this section, we discuss the periodicity of $\varpi _S$ under the action of $\Gamma_S$ on $\mathfrak{h}_S^n$. In other words, we investigate the nature of the Heegner object composed with the projection, $\pi \circ \varpi_S : T_S^1 \rightarrow \Gamma_S\backslash \mathfrak{h}_S^n$.

 We first define the notion of periods of the Heegner object $\varpi_S$ (with respect to $\Gamma_S$).
\begin{dfn}\label{periods}
A pair $(A,\rho)$ of $A \in \Gamma_S$ and $\rho \in T_S^1$ is said to be a period of $\varpi_S$ (with respect to $\Gamma_S$) if it satisfies
$$A\varpi_S (t) = \varpi_S (\rho t),\  \forall t \in T_S^1.$$
 Then, we simply say $A \in \Gamma_S$ (resp. $\rho \in T_S^1$) is a period of $\varpi_S$ assuming the existence of $\rho \in T_S^1$ (resp. $A \in \Gamma_S$) such that $(A,\rho)$ is a period.

We denote by $\Gamma_{\varpi_S} \subset \Gamma_S$ (resp. $\Lambda_{\varpi_S} \subset T_S^1$) the set of periods $A \in \Gamma_S$ (resp. $\rho \in T_S^1$) of $\varpi_S$. It is clear that they are both subgroups.
\end{dfn} 

Then the Heegner object composed with the projection, $\pi \circ \varpi_S $, factors through $T_S^1 / \Lambda_{\varpi_S}$, i.e.
$$
\begin{tabular}{ccc}
$T_S^1$&$\overset{\varpi _S}{\longrightarrow}$&$\mathfrak{h} _S^n$\\
$\downarrow$&$\circlearrowright$&$\downarrow \pi$\\
$T_S^1 / \Lambda_{\varpi_S}$&$\dashrightarrow$&$\Gamma _S \backslash \mathfrak{h} _S^n.$
\end{tabular}
$$

 For $A\in \Gamma_{\varpi_S}$, Proposition \ref{heeginj} assures that there exists a unique $\rho \in \Lambda_{\varpi_S}$ such that $(A,\rho)$ is a period. Thus we define a group homomorphism $\psi= \psi_{\varpi_S}: \Gamma_{\varpi_S} \rightarrow \Lambda_{\varpi_S} \subset T_S^1$ so that $(A, \psi(A))$ is a period of $\varpi_S$.

 We will show that, in fact, the period group $\Gamma_{\varpi_S}$ coincides with the unit group $\Gamma_{w,S} \simeq \mathcal{O}_{w,S}^1$.
 
Let $L: F \rightarrow T_S^1; \epsilon \mapsto (|v(\epsilon)|_p)_{v|p\in S}$ be the homomorphism taking absolute values, and denote also by $L$ the composition $L\circ \varphi: \Gamma_{w,S} \overset{\varphi}{\overset{\sim}{\rightarrow}} \mathcal{O}_{w,S}^1\subset F \overset{L}{\rightarrow} T_S^1$. Here, if $p=\infty$, then $|~|_{\infty}$ is the usual absolute value of $\mathbb{C}$, and if $p \in S^{\infty}$, then $|~|_p$ is the normalized $p$-adic absolute value: $|p|_p=p^{-1}$. 

\begin{prop}\label{unitisperiod}
 Let $A \in \Gamma_{w,S}$, and set $\epsilon := \varphi(A) \in F$. Then, 
  $$A \varpi _S (t) = \varpi _S (L(\epsilon)t),\ \  \forall t \in T_S^1.$$
 In other words, we have $\Gamma_{w,S} \subset \Gamma_{\varpi_S}$, and $L\circ \varphi = \psi$ on $\Gamma_{w,S}$.
\end{prop}
\begin{proof}
(1) At $\infty$, we get (in $\mathfrak{h}^n = GL_n(\mathbb{R}) / \mathbb{R}^{\times} O(n)$)
 \begin{eqnarray*}
A \varpi_{\infty} (t_1, \cdots, t_r) &=& [t_1 A w^{(1)}, \cdots, t_r A w^{(r)} ] \\ 
													  &=& [t_1 \epsilon ^{(1)} w^{(1)}, \cdots, t_r \epsilon^{(r)} w^{(r)} ] \\ 
													  &=& [t_1 |\epsilon ^{(1)}| w^{(1)}, \cdots, t_r |\epsilon^{(r)}| w^{(r)} ] \\ 
													  &=& \varpi_{\infty} (|\epsilon ^{(1)}|t_1, \cdots, |\epsilon ^{(r)}|t_r).
 \end{eqnarray*}
In the third equality, we use the fact that, for $\lambda \in \mathbb{C}^{\times}$ and $z=x+iy \in \mathbb{C}$ identified with $(y,x) \in \mathbb{R}^2$, there exists $R \in O(2)$ such that $\lambda (y,x) = \lambda z = |\lambda| (y,x)R$, i.e. complex multiplication decomposes into scaling and rotation. In the fourth equality we use $\prod_{\text{real}} |\epsilon ^{(i)}| \prod_{\text{complex}} |\epsilon ^{(i)}|^2=|N_{F/\mathbb{Q}}(\epsilon)|=|\det A| =1$.

(ii) At a finite place $p$, we get (in $\mathfrak{h}^n_p = GL_n(\mathbb{Q}_p) / \mathbb{Q}_p^{\times} GL_n(\mathbb{Z}_p)$)
 \begin{eqnarray*}
 A \varpi _p (t_1, \cdots, t_g) &=& [p^{e_1}t_1^{-1} v_1(A w), \cdots, p^{e_g}t_g^{-1} v_g(A w) ] \\ 
													  &=& [p^{e_1}t_1^{-1} v_1(\epsilon) v_1( w), \cdots, p^{e_g}t_g^{-1} v_g(\epsilon)v_g(w) ] \\ 
													  &=& [p^{e_1}t_1^{-1} |v_1(\epsilon)|_p^{-1} v_1w, \cdots, p^{e_g}t_g^{-1} |v_g(\epsilon)|_p^{-1} v_g(w) ] \\ 
													  &=& \varpi _p (|v_1(\epsilon)|_pt_1, \cdots, |v_g(\epsilon)|_pt_g).
 \end{eqnarray*}
 In the third equality we use the fact that, for $\lambda \in F_v^{\times}$ and $z \in F_v$ identified with an element in $\mathbb{Q}_p^{n_v} \simeq F_v $ (fixed in Section \ref{defheeg}), there exists $R \in GL_{n_v}(\mathbb{Z}_p)$ such that $\lambda z = |\lambda|_p^{-1} z R$. Here we need the assumption that $F/ \mathbb{Q}$ is unramified at $p$. In the fourth equality we use $\prod |v_i(\epsilon)|_p^{n_{v_i}} =|N_{F/\mathbb{Q}}(\epsilon)|_p=|\det A|_p =1$. 
\end{proof}

Next, we prove the following converse to Proposition \ref{unitisperiod}.

\begin{prop}\label{periodisunit}
Let $\varpi_S : T_S^1 \rightarrow \mathfrak{h}_S^n$ be the $S$-Heegner object associated to $w, (e_p)_p$ as in Definition \ref{dfnheeg}. Let $V \subset T_S^1$ be a subgroup satisfying the following condition: 
 there exists an archimedean place $\sigma_l$ such that for any $\epsilon \in \mathbb{R}_{>0}$, there exists $(t_{pv})_{p \in S,v|p} \in V$ such that $t_{\infty \sigma_i}/t_{\infty \sigma_l}<\epsilon$, $\forall i \neq l$. 

 If $A \in \Gamma_S$ satisfies 
 \begin{align*} A \varpi _S (t) = \varpi _S (\rho t),\ \  \forall t \in V, \tag*{$\diamondsuit$}
 \end{align*} 
 for some $\rho \in T_S^1$, then $A \in \Gamma_{w,S}$ and $\rho =L \circ \varphi(A)$.
\end{prop}

\begin{proof}
 Let $\varpi_{\infty}$ be the archimedean part of the Heegner object $\varpi_S$. We focus only on the archimedean part. We write $t_i$ for $t_{\infty \sigma_i}$ as before. Set $V_{\infty} := \{(t_{\infty v}) \in (\prod_{v|\infty}\mathbb{R}_{>0})^1\ |\ (t_{pv})_{p \in S, v|p } \in V\}$, and $\rho=(\rho_i)_i :=(\rho_{\infty \sigma_i})_i \in (\prod_{v|\infty}\mathbb{R}_{>0})^1$.
 
 As in the proof of Proposition \ref{heeginj}, set $W := (w^{(1)} \cdots w^{(r)}) \in GL_n(\mathbb{R})$, so that $\varpi (t) =\varpi_{\infty} (t_1,\dots, t_r) = [ W
I(t)
]$, where $I(t)=I(t_1,\dots, t_r) :=
 \left(
		\begin{array}{ccc}
		t _1I_{n_1}&& \\
		&\ddots& \\
		&&t_rI_{n_r}
		\end{array}
\right)$. Then, by the archimedean part of the identity $\diamondsuit$, there exists a map $R : V_{\infty} \rightarrow O(n)$ such that 
$AWI(t) = WI(\rho t) R(t)$ for all $t \in V_{\infty}$. In particular, $R(t) = I(\rho ^{-1})I(t^{-1})W^{-1}AWI(t) \in O(n)$. Define $B_{ij} \in M_{n_i n_j}(\mathbb{R})$ by $W^{-1}AW = (B_{ij})_{ij}$. Since $R(t){}^t\!R(t) = I_n$, we have
$$
\left\{ 
\begin{array}{ll}
i= j \Rightarrow		 & \sum_{k=1}^{r}\frac{t_k^2}{\rho_i^2 t_i^2}B_{ik}{}^t\!B_{ik} = I_{n_i}\\
i\neq j \Rightarrow & \sum_{k=1}^{r}\frac{t_k^2}{\rho_i \rho_j  t_i t_j}B_{ik}{}^t\!B_{jk} = 0.
\end{array}\right.
$$
By the condition, we can take $t_k/t_l (k\neq l)$ arbitrarily small. Therefore, by setting $i=l$ and separating the terms $k=l$ and $k\neq l$, we conclude 
$B_{ll}{}^t\!B_{ll} =\rho_l^2I_{n_l}$ and $B_{ll}{}^t\!B_{jl} = 0$ for all $j \neq l$. In particular, $B_{ll}/\rho_l \in O(n_l)$ and $B_{jl}=0$ for $j \neq l$. Then we have $A w^{(l)} = \rho_l w^{(l)}B_{ll}/\rho_l = \lambda w^{(l)}$ for some $\lambda \in F_{\sigma _l}^{\times}$. Here, if $\sigma_l$ is a complex place, we use the fixed identification $F_{\sigma _l}=\mathbb{C} \simeq \mathbb{R}^2$. Therefore $A \in \Gamma_{w,S}$ since $\lambda \in \sigma_l (F)$ is automatic. By Proposition \ref{unitisperiod}, we see $\rho =L \circ \varphi (A)$.
\end{proof}

The condition in Proposition \ref{periodisunit} is always satisfied for $V=T_S^1$ by taking $\sigma _1$ for $\sigma_l$. Thus,
\begin{cor}\label{corunitperiod}
We have $\Gamma_{w,S}= \Gamma_{\varpi_S}$ and $L \circ \varphi =\psi$. In particular, $\Lambda_{\varpi_S} = L(\mathcal{O}_{w,S}^1)$. \qed
\end{cor}

 The next theorem is a generalization of geodesic Lagrange's theorem.
\begin{thm}\label{perHeeg} 
The quotient $T_S^1 / \Lambda_{\varpi_S}$ is compact with its natural topology. In particular, the image of $\pi \circ \varpi _S$ in $SL_n(\mathbb{Z}[1/N]) \backslash \mathfrak{h}_S^n$ is compact. 
\end{thm}

\begin{proof}
Let $\mathbb{A}_F^{\times}$ be the group of id\`eles of $F$. Let $\mathbb{A}^1_F=\{x=(x_v)_v\in \mathbb{A}^{\times}_F\ |\ \prod_v|x_v|_v=1\} \supset V=\{x=(x_v)_v \in \mathbb{A}_F  |\ \prod_{v|p}|x_v|_v=1,\forall p\in S,\ |x_v|_v=1,\forall v|p \not\in S\}$ be its subgroups. Here, for a place $p$ of $\mathbb{Q}$ and a place $v|p$ of $F$, $|x_v|_v:= |x_v|_p^{n_v}$, where $n_v=[F_v:\mathbb{Q}_p]$.
We claim that $V$ is an open subgroup of $\mathbb{A}^1_F$. This is because, for a prime $p \in S^{\infty}$, $\prod_{v|p}|x_v|_v$ takes value in a discrete subgroup of $\mathbb{R}_{>0}$, hence $\prod_{v|p}|x_v|_v=1$ for all $p\in S$ and $|x_v|_v=1$ for all $v|p \not\in S$ is an open condition.
Since $\mathbb{A}_F^1/F^{\times}$ is compact group and $V$ is an open subgroup of $\mathbb{A}^1_F$, $V/V\cap F^{\times}$ is compact. 
Now, since $\mathcal{O}_{w,S}$ is an order in $F$ (Proposition \ref{orderunits}), $\mathcal{O}_{w,S}^1$ is a subgroup of finite index of $V\cap F^{\times}=\mathcal{O}_F[1/N]^1$, where $\mathcal{O}_F$ is the ring of integers of $F$, and $\mathcal{O}_F[1/N]^1:=\{x \in\mathcal{O}_F[1/N]^{\times}\ |\ N_{F/\mathbb{Q}}(x)=1\}$. 
On the other hand, we have a surjective homomorphism $V \rightarrow T^1_S ; (x_v)_v \mapsto (|x_v|_p)_{v|p\in S}$, and Corollary \ref{corunitperiod} implies that the image of $\mathcal{O}_{w,S}^1 \subset V\cap F^{\times}$ is $L(\mathcal{O}_{w,S}^1)=\Lambda_{\varpi_S}$. Therefore $T^1_S/ \Lambda_{\varpi_S}$ is compact.
\end{proof}

\begin{rmk}
In fact, only the inclusion $\Gamma_{w,S} \subset \Gamma_{\varpi_S}$ (Proposition \ref{unitisperiod}) is enough to prove Theorem \ref{perHeeg}.
\end{rmk}

\subsection{$\chi$-component}\label{sectchi}
 Here we consider the ``$\chi$-component" of Heegner objects for a quadratic character $\chi$. This enables us to treat the $\chi$-component of unit groups.
 
 Let $S$ be a finite set of places of $\mathbb{Q}$ including the archimedean place $\infty$. Let $F'/ \mathbb{Q}$ be an extension of degree $d$, and let $F/F'$ be a quadratic extension with Galois group $G=\Gal(F/F')$. We denote by $S^{\infty}$ the set of finite places in $S$, by $S_F$ (resp. $S_{F'}$) the set of places of $F$ (resp. $F'$) above $S$, and by $S_{F,p}$ (resp. $S_{F',p}$) the set of places of $F$ (resp. $F'$) above $p \in S$. We denote by $G_{v'}$ the decomposition group of $v' \in S_{F'}$ in $F/F'$. We assume that $F/\mathbb{Q}$ is unramified at every finite place $p \in S^{\infty}$, and at least one infinite place $v' \in S_{F',\infty}$ splits in $F/F'$.
 
  Let $w={}^t\!(w_1, \dots, w_n) \in F^n$ be a basis of $F$ over $\mathbb{Q}$, where we set $n=2d$, and let $\varpi_S : T_S^1 \rightarrow \mathfrak{h}^n_S$ be the Heegner object associated to $w$ (with any index $(e_p)_p$) as in Section \ref{defheeg}. We keep the notations $\mathcal{A}_{w,S}, \Gamma_{w,S}, \varphi , \mathcal{O}_{w,S} , \Gamma_{\varpi_S}$, etc. in Sections \ref{sectorderunits} and \ref{sectper}.

 Let $\chi : G \rightarrow \{\pm 1\}$ be the non-trivial character. Recall that, for a $\mathbb{Z}[G]$-module $M$, we define its $\chi$-component as $M^{\chi} := \{m \in M\ |\ s(m) = \chi (s) m \}$.
 
 \begin{dfn}\label{defchiheeg}~
 \begin{enumerate}[(i)]
 \item The Galois group $G$ acts naturally on $T_S^1 = (\prod_{v \in S_{F,\infty}}\mathbb{R}_{>0})^1\times \prod_{p\in S^{\infty}}(\prod_{v \in S_{F,p}}p^{\mathbb{Z}})^1$ by permuting the indices $v \in S_F$. The $\chi$-component $(T_S^1)^{\chi}$ of $T_S^1$ is defined by
  $$(T_S^1)^{\chi} = \{t=(t_{v})_{v\in S_F} \in T_S^1\ |\ s(t)_v(=t_{s^{-1}v})=t_v^{\chi(s)}, \forall s \in G \}.$$
  Then we define the $\chi$-component of the Heegner object as the restriction of $\varpi_S$ to $(T_S^1)^{\chi}$: 
  $$\varpi _S^{\chi}:= \varpi_S|_{(T_S^1)^{\chi}} : (T_S^1)^{\chi} \rightarrow \mathfrak{h}^n_S .$$ 
\item Let $N:= \prod_{p\in S,p \neq \infty} p$ as in Section \ref{sectorderunits}. Define $\Gamma_{\varpi_S}(\chi)$ as, 
 $$\Gamma_{\varpi_S}(\chi) := \{A \in SL_n(\mathbb{Z}[1/N]) \ |\ \exists \rho \in (T_S^1)^{\chi}, A\varpi _S^{\chi}(t) = \varpi _S^{\chi}(\rho t) , \forall t \in (T_S^1)^{\chi}  \}.$$
\end{enumerate}
\end{dfn}
 The following proposition asserts that $\Gamma_{\varpi_S}(\chi)$ is actually the $\chi$-component of $\mathcal{O}_{w,S}^{1}\simeq \Gamma_{w,S}$ modulo torsion part. For an additive group $M$, we denote by $M_{\rm tor}$ the subgroup of torsion elements, and denote by $M/({\rm tor})$ the quotient of $M$ by $M_{\rm tor}$.
\begin{prop}\label{chiunitperiod}
\begin{enumerate}[{\rm (i)}]
\item $(\Gamma_{w,S})_{\rm tor} \subset \Gamma_{\varpi_S}(\chi) \subset \Gamma_{w,S}$.
\item The following diagram commutes: 
$$
\begin{tabular}{ccc}
$\Gamma_{w,S}/({\rm tor})$&$\overset{\varphi}{\overset{\sim}{\longrightarrow}}$&$\mathcal{O}_{w,S}^1/({\rm tor})$ \\
\rotatebox{90}{$\subset$}&&\rotatebox{90}{$\subset$} \\
$\Gamma_{\varpi_S}(\chi)/({\rm tor})$&$\overset{\sim}{\longrightarrow}$&$(\mathcal{O}_{w,S}^1/({\rm tor}))^{\chi}.$\\
\end{tabular}
$$
Note that $\mathcal{O}_{w,S}^1/({\rm tor})$ has a natural $\mathbb{Z}[G]$-module structure, since the subgroup of torsion elements is $G$-stable.
\end{enumerate}
\end{prop}
\begin{proof}
(i) For $A \in (\Gamma_{w,S})_{\rm tor}$, we have $|v(\varphi(A))|_p=1$ for all $p \in S$ and $v \in S_{F,p}$. Thus we get the first inclusion by Proposition \ref{unitisperiod}.

 Since we have assumed that at least one infinite place of $F'$ splits in $F/F'$, say $v' \in S_{F',\infty}$, the condition in Proposition \ref{periodisunit} is satisfied for $V=(T_S^1)^{\chi}$. Indeed, if we denote by $v_1,v_2 \in S_F$ the two places above $v'$, then $(T_S^1)^{\chi}$ contains elements of the form $(t_v)_{v \in S_F}$ such that $t_{v_1}=t, t_{v_2}=1/t$ and $t_{v}=1 ~(v \neq v_1,v_2)$ for $t \in \mathbb{R}_{>0}$. Thus we get the second inclusion $\Gamma_{\varpi_S}(\chi) \subset \Gamma_{w,S}$ by Proposition \ref{periodisunit}. 

(ii) Let $L: \mathcal{O}_{w,S}^1/({\rm tor}) \rightarrow T_S^1$ be the ``absolute values" map defined in Section \ref{sectper}. Taking into account that this is an injective $\mathbb{Z}[G]$-homomorphism, we see $(\mathcal{O}_{w,S}^1/({\rm tor}))^{\chi} = L^{-1}((T_S^1)^{\chi})$. On the other hand,
 by Proposition \ref{periodisunit} and the inclusion $\Gamma_{\varpi_S}(\chi) \subset \Gamma_{w,S}$, we see that
 $\Gamma_{\varpi_S}(\chi) = \{A  \in \Gamma_{w,S}\ |\ L\circ \varphi (A) \in (T_S^1)^{\chi} \} = (L \circ \varphi)^{-1} ((T_S^1)^{\chi})$. This proves (ii).
\end{proof}

Let $\Lambda = \Lambda_{\varpi_S}= L(\mathcal{O}_{w,S}^1)$ be the image of the ``absolute values" map $L$ defined in Section \ref{sectper}. 
Then $\Lambda$ is a $\mathbb{Z}[G]$-submodule of $T_S^1$. Let $\Lambda^{\chi}$ be the $\chi$-component of $\Lambda$. 
The next theorem is the periodicity result for $\varpi_S^{\chi}$.
\begin{thm}\label{chiperiodicity}
The quotient $(T_S^1)^{\chi}/\Lambda^{\chi}$ is compact. In particular, the image of $\pi \circ \varpi_S^{\chi}$ in $SL_n(\mathbb{Z}[1/N]) \backslash \mathfrak{h}_S$ is compact. 
\end{thm} 
\begin{proof}
We first embed $T_S^1$ into $\bigoplus_{v\in S_F} \mathbb{R}$ by $T_S^1 \hookrightarrow \bigoplus_{v\in S_F} \mathbb{R} ; (t_v)_v \mapsto (\log (t_v))_v$. We equip $\bigoplus_{v\in S_F} \mathbb{R}$ with a natural $G$-action by permuting the indices $v \in S_F$ so that this embedding is a $G$-homomorphism. Let $V:=\{(x_v)_v \in \bigoplus_{v\in S_F} \mathbb{R}\ |\ \sum_{v|p}n_v x_v=0, \forall p \in S\}$, the $G$-subspace. Then, clearly, the embedding $T_S^1 \hookrightarrow \bigoplus_{v\in S_F} \mathbb{R}$ factors through $V$. We have the following inclusions of $G$-modules:
$$
\begin{tabular}{ccccc}
$\Lambda$&$\hookrightarrow$&$T_S^1$&$\hookrightarrow$&$V$ \\
\rotatebox{90}{$\subset$}&&\rotatebox{90}{$\subset$}&&\rotatebox{90}{$\subset$} \\
$\Lambda^{\chi}$&$\hookrightarrow$&$(T_S^1)^{\chi}$&$\hookrightarrow$&$V^{\chi}$ \\
\end{tabular}
$$
Furthermore, this diagram is a pull-back, i.e. $(T_S^1)^{\chi} = T_S^1 \cap V^{\chi}$ and $\Lambda^{\chi} = \Lambda \cap V^{\chi}$ in $V$. 
By Proposition \ref{orderunits} and Theorem \ref{perHeeg}, we know that $T_S^1/\Lambda$ is compact, and it follows that $\Lambda$ is a lattice in $V$.
We claim that $\Lambda ^{\chi}$ is a lattice in $V^{\chi}$. The discreteness of $\Lambda$ in $V$ implies that of $\Lambda^{\chi}$ in $V^{\chi}$. It suffices to show that $\Lambda^{\chi}$ generates $V^{\chi}$ over $\mathbb{R}$. To see this, consider the map $e : V \rightarrow V; x \mapsto \sum_{s \in G} \chi(s^{-1})s(x)$. Then $e$ maps $V$ (resp. $\Lambda$) to $V^{\chi}$ (resp. $\Lambda^{\chi}$), and $e(x)= (\# G)x$ for $x \in V^{\chi}$, where $\#G=2$. Therefore $e(\Lambda) \subset \Lambda^{\chi}$ generates $e(V^{\chi})=2V^{\chi}=V^{\chi}$. 
Now, $(T_S^1)^{\chi}/\Lambda^{\chi} = T_S^1/\Lambda \cap V^{\chi}/\Lambda^{\chi}$ in $V/\Lambda$, and both $ T_S^1/\Lambda$ and $V^{\chi}/\Lambda^{\chi}$ are compact . Therefore $(T_S^1)^{\chi}/\Lambda^{\chi}$ is compact.
\end{proof}

\begin{rmk}
We use the assumption that at least one infinite place of $F'$ splits in $F/F'$ only in the proof of Proposition \ref{chiunitperiod} to show the inclusion $\Gamma_{\varpi_S}(\chi) \subset \Gamma_{w,S}$. Theorem \ref{chiperiodicity} holds true without this assumption.
\end{rmk}

\section{Geodesic continued fractions}\label{gcf}
%
 In this section (and the next), we develop the theory of geodesic continued fractions, which was originally studied to find good Diophantine approximations, to determine the group of periods $\Gamma_{\varpi_S}$ (or $\Gamma_{\varpi_S}(\chi)$) of Heegner objects of number fields with rank one ($\chi$-)unit group. 

 First, we present a general theory of geodesic continued fractions. We introduce a new notion of weak convexity condition instead of the convexity condition in the previous works, in order to treat a larger class of geodesics including one dimensional ($\{\infty\}$-)Heegner objects. The most of the expositions in Sections \ref{gcf1} and \ref{siegelsets}, except for Theorem \ref{geodalgo}, are generalizations of Lagarias~\cite{lagarias} and Beukers~\cite{beukers}. In fact, in the case where $\mathfrak{G}= \mathfrak{G}_{(r, n-r), (1, 0)}$ (resp. $\mathfrak{G}_{(1,n-1),(1,0)}$) (see Definition \ref{dfnflatgeod}), and the Minkowski fundamental domain (resp. the LLL-reduced set) is chosen as a fundamental domain (see Definition \ref{FD}),  essentially the same algorithm as Algorithm \ref{GCF} is presented in \cite{lagarias} (resp. \cite{beukers}). However they do not treat algebraic numbers, and the periodicity of the algorithm is not discussed there.
 On the other hand, Theorem \ref{geodalgo}, which is simple and almost trivial, is unique in this paper, and combined with the periodicity of Heegner objects (Theorems \ref{perHeeg} and \ref{chiperiodicity}), assures us the periodicity of the algorithm applied to Heegner objects, and provides us with a non-torsion ($\chi$-)units of number fields (Theorems \ref{arch-main} and \ref{chiarch-main}).

\subsection{The algorithm}\label{gcf1} 
  Let $\mathfrak{h}^n$ be the generalized upper half space $GL_n(\mathbb{R}) / \mathbb{R}^{\times} O(n)$ defined in Section \ref{defheeg}. 
Set $\Gamma = SL_n(\mathbb{Z})$.
 The following theorem about the geodesics on $\mathfrak{h}^n$ is known, though we do not use this fact. See Terras~\cite[pp.12--17, Theorem 1]{terras2}.
 \begin{thm}[Geodesics on $\mathfrak{h}^n$]\label{terrasgeo}
 
  Let $V$ be a one-parameter subgroup of the positive entry diagonal subgroup of $GL_n(\mathbb{R})$. Then any (left) $GL_n(\mathbb{R})$-translation of the image of $V$ in $\mathfrak{h}^n$ is a geodesic with respect to the natural Riemannian metric. Conversely, every geodesic on $\mathfrak{h}^n$ is of this form. \qed
 \end{thm}
 
 
 \begin{ex} One dimensional $\{\infty \}$-Heegner objects are geodesics.
 \end{ex}
 
 The essence of geodesic continued fraction algorithm is to observe ``a scenery" along the geodesic. In the following we explain what ``a scenery" means.
 
  We first recall the notion of a fundamental domain for the discrete action of $\Gamma$ on $\mathfrak{h}^n$ with an additional condition of ``geodesic convexity''.
 \begin{dfn}[Fundamental domain, weak $\mathfrak{G}$-convexity]\label{FD}~
 
  Let $\mathfrak{G}$ be a family of geodesics on $\mathfrak{h}^n$ which is stable under the $GL_n(\mathbb{R})$ translation. We call a geodesic in $\mathfrak{G}$ a $\mathfrak{G}$-{\it geodesic}.
  
 A {\it weakly $\mathfrak{G}$-convex fundamental domain} for $\Gamma = SL_n(\mathbb{Z})$ is the closed subset $\mathcal{F} \subset \mathfrak{h}^n$ satisfying the following properties.
 \begin{enumerate}[{\rm (i)}]
 \item {\it tiling property}: $$\mathfrak{h}^n = \bigcup\limits_{\gamma \in \Gamma} \gamma \mathcal{F}.$$
 \item {\it finiteness}: there exists an open neighborhood $U$ of $\mathcal{F}$ such that $\{\gamma \in \Gamma\ |\ \gamma U \cap U \neq \emptyset\}$ is a finite set.
 \item {\it weak $\mathfrak{G}$-convexity}: for any $\mathfrak{G}$-geodesic $\varpi: \mathbb{R}_{>0} \rightarrow \mathfrak{h}^n$, $ \varpi ^{-1}(\varpi (\mathbb{R}_{>0}) \cap \mathcal{F})$ is a union of a finite number of intervals.
 \end{enumerate}
 If the property {\rm (iii)} can be replaced with the following, then $\mathcal{F}$ is said to be $\mathfrak{G}$-{\it convex}.
 \begin{enumerate}[{\rm (iv)}]
 \item For any $\mathfrak{G}$-geodesic $\varpi: \mathbb{R}_{>0} \rightarrow \mathfrak{h}^n$, $ \varpi ^{-1}(\varpi (\mathbb{R}_{>0}) \cap \mathcal{F})$ is an interval.
 \end{enumerate}
 \end{dfn}
 
  For the time being let us  fix $\mathfrak{G}$ and a weakly $\mathfrak{G}$-convex fundamental domain $\mathcal{F}$. Let $\varpi$ be a $\mathfrak{G}$-geodesic. Then for a point $\varpi (t)$ ($t \in \mathbb{R}_{>0}$) on the geodesic $\varpi$, we consider  to which translation $\gamma \mathcal{F}$ of $\mathcal{F}$ it belongs. (``A scenery".)
  
  Now we present the algorithm. 
  In the following, for $I \subset \mathbb{R}$ and $x\in I$, we denote by $[I]_x$ the connected component of $I$ containing $x$.
 \begin{algo}[Geodesic continued fraction algorithm]\label{GCF}~
 
  Let $\varpi : \mathbb{R}_{>0} \rightarrow \mathfrak{h}^n$ be a $\mathfrak{G}$-geodesic. 
 \begin{list}{}{}
 \item[{\it Preparation}:] Set $u_0=1$. Take $B_0 \in \Gamma$ such that $\varpi (1) \in B_0 \mathcal{F}$ and set $\varpi _0 (t) = B_0{}^{-1} \varpi (t) $.
 \item[{\it Forward loop}:] For $\varpi _k$ ($k\geq0$) we define $\varpi _{k+1}$ as follows.
 	\begin{itemize} 
	\item Set $t_k:= \sup [\varpi_k^{-1} (\varpi _k (\mathbb{R}_{>0}) \cap \mathcal{F})]_{u_k}$, $s_k:= \inf [\varpi_k^{-1} (\varpi _k (\mathbb{R}_{>0}) \cap \mathcal{F})]_{u_k}$.
	
If $t_k = \infty$, then set $s_{k'+1}=t_{k'}:=\infty, A_{k'}:=I_n, B_{k'}:=B_k, \varpi_{k'}:=\varpi_k$ for $k' \geq k$ and stop the forward loop. Here $I_n$ is the $n\times n$ identity matrix.
	\item Determine $A_{k+1} \in \Gamma$ and $u_{k+1} > t_k$ such that $\varpi _k([t_k, u_{k+1}]) \subset A_{k+1} \mathcal{F}$. 
	\item Define $\varpi _{k+1} := A_{k+1}{}^{-1} \varpi _k$, $B_{k+1} := B_{k}A_{k+1}$.
	\end{itemize}
 \item[{\it Backward loop}:] For $\varpi _k$ ($k\leq0$) we define $\varpi _{k-1}$ as follows.
 	\begin{itemize}
	\item Set $t_k:= \sup [\varpi_k^{-1} (\varpi _k (\mathbb{R}_{>0}) \cap \mathcal{F})]_{u_k}$, $s_k:= \inf [\varpi_k^{-1} (\varpi _k (\mathbb{R}_{>0}) \cap \mathcal{F})]_{u_k}$.
	
If $s_k = 0$, then set $s_{k'}=t_{k'-1}:= 0, A_{k'}:=I_n, B_{k'}:=B_k, \varpi_{k'}:=\varpi_k$ for $k' \leq k$ and stop the backward loop. Here $I_n$ is the $n\times n$ identity matrix.
	\item Determine $A_{k} \in \Gamma$ and $u_{k-1} < s_k$ such that $\varpi _k([u_{k-1}, s_k]) \subset A_{k}^{-1} \mathcal{F}$.
	\item Define $\varpi _{k-1} := A_{k} \varpi _k$, $B_{k-1} := B_kA_{k}^{-1}$.
	\end{itemize} 
 \end{list}
 \end{algo}

\begin{rmk}
\begin{enumerate}
\item We may start with another choice of $u_0 \in \mathbb{R}_{>0}$.
\item The choice of $A_k$ may no be unique in general. 
\item The algorithm does not depend on the choice of $u_k$, thus we may forget $u_k$ after we take $A_{k+1}$. 
\end{enumerate}
\end{rmk}

 \begin{prop}\label{gcfwelldef} 
 \begin{enumerate}[{\rm (i)}]
 \item Algorithm \ref{GCF} is well defined.
 \item $\lim_{k \to \infty} t_k = \infty$.
 \item $\lim_{k \to -\infty} s_k = 0$.
 \end{enumerate}
 \end{prop}
 \begin{proof}(i) Since the proofs are the same for forward and backward loop, we consider only forward loop.
 Only the existence of such $u_{k+1}$ and $A_{k+1}$ in the algorithm is non-trivial. Take an open neighborhood $U$ of $\mathcal{F}$ such that $\{\gamma \in \Gamma\ |\ \gamma U \cap U \neq \emptyset\}$ is a finite set. Then $\{\gamma \in \Gamma\ |\ \gamma \mathcal{F} \cap U \neq \emptyset\} =: \{\gamma_1, \dots, \gamma_N\}$ is also a finite set. Since we have $U \subset \bigcup_{i=1}^{N}\gamma _i\mathcal{F}$ by the tiling property, $\bigcup_{i=1}^N \varpi_{k} ^{-1}(\varpi_{k} (\mathbb{R}_{>0}) \cap \gamma_i\mathcal{F})$ is a neighborhood of $t_k$. On the other hand, by the weak $\mathfrak{G}$-convexity, each  $\varpi_{k} ^{-1}(\varpi_{k} (\mathbb{R}_{>0}) \cap \gamma_i\mathcal{F})$ is a union $\bigcup_{j}I_{ij}$ of a finite number of closed intervals $I_{ij}\subset \mathbb{R}_{>0}$. 
 Therefore there exists $u_{k+1}>t_k$ such that $[t_k, u_{k+1}]$ is contained in one of these closed intervals $I_{ij}$. Then, $\varpi_k([t_k,u_{k+1}]) \subset \gamma_i \mathcal{F}$. This proves (i).
 
  (ii) Since the proofs of {\rm (ii)} and {\rm (iii)} are the same, we prove only (ii). We keep the notations in the proof of (i).
Suppose $\{t_k\}_k$ is bounded from above. Then $t_k$ converges to a constant $t_{\infty} \in \mathbb{R}_{>0}$. Take $\gamma_{\infty} \in \Gamma$ such that $\varpi(t_{\infty}) \in \gamma_{\infty} \mathcal{F}$. Since $\gamma_{\infty}U \subset  \bigcup_{i=1}^{N} \gamma_{\infty} \gamma _i\mathcal{F}$, there exists some $i$ such that $\gamma_{\infty} \gamma _i\mathcal{F}$ contains infinitely many $\varpi(t_k)$. On the other hand, by the weak convexity property, we can find an increasing sequence $k_l \in \mathbb{N}$ such that $B_{k_l}$ are distinct and $\varpi(t_{k_l}) \in \gamma_{\infty} \gamma _i\mathcal{F}$. Then, for large enough $l$, $\varpi(t_{k_l}) \in B_{k_l}\mathcal{F} \cap \gamma_{\infty} \gamma _i \mathcal{F}$. This contradicts the finiteness property.
 \end{proof}
 
 The next theorem interprets the periodicity of geodesics into that of the algorithm. 
 
 \begin{thm}\label{geodalgo}
   Let $\varpi : \mathbb{R}_{>0} \rightarrow \mathfrak{h}^n$ be a $\mathfrak{G}$-geodesic. Suppose $\varpi$ is periodic under the action of $\Gamma$, i.e. there exist $\epsilon \in \mathbb{R}_{>0}, \epsilon \neq 1$, and a non-torsion element $A \in \Gamma$ such that $A \varpi (t) =\varpi (\epsilon t)$ for all $t \in \mathbb{R}_{>0}$. 
\begin{enumerate}[{\rm (i)}]   
\item Then Algorithm \ref{GCF} applied to $\varpi$ does not stop, i.e. $s_{k}, t_{k} \in \mathbb{R}_{>0}$ for all $k \in \mathbb{Z}$.
\item Let $\{(A_k, B_k, \varpi _k)\}_k$ be the sequence taken by the algorithm. Then there exist $M, N \in \mathbb{N}, M<N$, and $\rho \in \mathbb{R}_{>1}$ such that 
   $\varpi _M (t) = \varpi _N (\rho t)$. 
\end{enumerate}   
 \end{thm}
 \begin{proof} We can assume without loss of generality that $\epsilon >1$.
 
 (i) It suffices to prove that, if $\varpi(u) \in B\mathcal{F}$ $(u \in \mathbb{R}_{>0}, B \in \Gamma)$, then there exists $u' \in \mathbb{R}_{>0}$ such that $u<u'$ and $\varpi(u') \not\in B \mathcal{F}$. Indeed, since $A$ is not a torsion element, there exists $N \in \mathbb{N}$ such that $B\mathcal{F} \cap A^NB\mathcal{F} = \emptyset$. Then $\varpi(\epsilon^Nu)= A^N\varpi(u) \not\in B\mathcal{F}$.
 
 (ii) Take any $u_0 \in [s_0, t_0]$ and set $u_k = \epsilon ^k u_0$. For each $u_k$, Proposition \ref{gcfwelldef} allows us to take $N_k \in \mathbb{N}$ such that $u_k \in [s_{N_k}, t_{N_k}]$ and $N_k$ is monotonic in $k$. 
 
 Then we have $\varpi (u_k) = A^k \varpi (u_0) \in A^k B_0\mathcal{F} \cap B_{N_k} \mathcal{F}$. Therefore $\varpi (u_0) \in B_0\mathcal{F} \cap A^{-k}B_{N_k} \mathcal{F}$. Now, because of the finiteness property of $\mathcal{F}$, there exist $k,l \in \mathbb{N}, k<l$, such that $A^{-k}B_{N_k} = A^{-l}B_{N_l}$. Then, $\varpi _{N_k} (t) = B_{N_k}{}^{-1} \varpi (t) =B_{N_l}{}^{-1} A^{l-k} \varpi (t) = \varpi _{N_l} (\epsilon ^{l-k} t)$. Furthermore, since $A$ is a non-torsion element, $N_k \neq N_l$. This completes the proof. 
 \end{proof}
 
 \begin{rmk}
 In fact, the assumption that $A$ is a non-torsion element is superfluous. We can show using Theorem \ref{terrasgeo} and the similar argument as in Proposition \ref{heeginj} that $\varpi$ is injective. Then the assumption $\epsilon \neq 1$ implies that $A$ is a non-torsion element.
 \end{rmk}
 
\subsection{Siegel sets and LLL-reducedness}\label{siegelsets}
 Next we explicitly define some fundamental domains. First we recall the Iwasawa decomposition for $GL_n(\mathbb{R})$. 
 
\begin{thm}[Iwasawa decomposition for $GL_n(\mathbb{R})$]~

 For any $g \in GL_n(\mathbb{R})$ there exist $$X(x_{ij}) := 
 \left(
\begin{array}{ccc}
1&&x_{ij}\\
&\ddots&\\
\hsymb{0}&&1
\end{array}
\right),
\  Y(r_i) := 
\left(
\begin{array}{ccc}
r_1&&\hsymb{0}\\
&\ddots&\\
\hsymb{0}&&r_n
\end{array}
\right),\text{ and } 
R\in O(n)
$$where $x_{ij} \in \mathbb{R},\  r_i > 0$, such that $g=XYR$. Furthermore, this decomposition is unique. 
\end{thm} 
\begin{proof}
See \cite[Proposition 1.2.6]{goldfeld}. 
\end{proof}

\begin{rmk}
 Note that any $z \in \mathfrak{h}^n$ can be written uniquely in the form $z=[X(x_{ij})Y_1(r_i)]$ where $Y_1(r_i)=Y(r_i,r_n=1)$.  
\end{rmk}

\begin{dfn}[Siegel set, LLL-reduced set]\label{dfnsiegelsets}~
\begin{enumerate}
\item For $a,b > 0$ we define the $Siegel$ $set$ $\Sigma_{a,b}$ by
$$\Sigma_{a,b} := \{z=[X(x_{ij})Y_1(r_i)] \in \mathfrak{h}^n \ |\  |x_{ij}| \leq b,\  r_i/r_{i+1} \geq a  \}.$$
In particular, we set $\mathcal{S} := \Sigma_{\frac{\sqrt{3}}{2},\frac{1}{2}}$. 
\item For a fixed $\omega \in [3/4,1]$, we define the $LLL$-$reduced$ $set$ $\mathcal{L} = \mathcal{L}_{\omega}$ by
$$\mathcal{L}_{\omega} := \{z=[X(x_{ij})Y_1(r_i)] \in \mathfrak{h}^n\  |\  |x_{ij}| \leq \frac{1}{2},\  \omega r_{i+1}^{\ 2} \leq r_i^{\ 2} + x_{i, i+1}^{\ 2} r_{i+1}^{\ 2}  \}.$$
\end{enumerate}
\end{dfn} 
See, for example, \cite{borel} or \cite{goldfeld}  for detailed accounts of Siegel sets, and \cite{beukers}, \cite{cohen} or \cite{LLL} for LLL-reduced sets.

\begin{prop}\label{siegelsetFD}
The sets $\mathcal{S}$ and $\mathcal{L}$ are fundamental domains, i.e. they satisfy {\rm (i)} and {\rm (ii)} in Definition \ref{FD}.
\end{prop}
\begin{proof}
 It is known that $\mathfrak{h}^n = \bigcup_{\gamma \in \Gamma} \gamma \Sigma_{a,b}$ for $a \geq \sqrt{3}/2, b \geq 1/2$, and $\{\gamma \in \Gamma \ |\ \gamma \Sigma_{a,b} \cap \Sigma_{a,b}\}$ is a finite set for any $a,b>0$. See \cite[Proposition 1.3.2]{goldfeld} , \cite[Th\'eor\`eme 1.4, Th\'eor\`eme 4.6]{borel}. Now 
 \begin{eqnarray*}{}
 \mathcal{S} \subset \Sigma_{\sqrt{\omega} ,1/2} \subset \mathcal{L}_{\omega} &\subset& \Sigma_{\sqrt{\omega-1/4},1/2} \subset \Sigma_{1/2,1/2},\\
 \Sigma_{1/2,1/2} &\subset& (\Sigma_{1/3,1})^{\circ},
 \end{eqnarray*}
 where, for a topological space $X$, $X^{\circ}$ denotes the interior of $X$. This proves the proposition.
\end{proof}

 Next we would like to show that $\mathcal{S}$ and $\mathcal{L}$ are ``convex" fundamental domains. To do this, we introduce certain classes of geodesics. 
 
\begin{dfn}\label{dfnflatgeod} For a partition $\lambda$ of $n$; $\lambda = (n_1,\dots, n_r) \in \mathbb{N}^r, n=n_1+\dots +n_r, 1\leq n_1, \dots, n_r \leq n$, set
$$I_{\lambda}(t_1,\dots,t_r) = 
 \left(
				\begin{array}{ccc}
				t_1I_{n_1}&&\\
				&\ddots& \\
				&&t_r I_{n_r}
				\end{array}
				\right)
				,$$ 
where $I_{n_i}$ is the $n_i \times n_i$ identity matrix. Let $\lambda = (n_1,\dots, n_r)$ be a partition of $n$, and let $d = (d_1,\dots, d_r) \in \mathbb{N}^r$. We denote by $\mathfrak{G}_{\lambda, d}$ the set of geodesics of the form,
$$\varpi : \mathbb{R}_{>0} \rightarrow \mathfrak{h}^n; t \mapsto [gI_{\lambda}(t^{d_1},\dots, t^{d_r})],$$
for $g \in GL_n(\mathbb{R})$. Then, it is clear that $\mathfrak{G}_{\lambda, d}$ is stable under the action of $GL_n(\mathbb{R})$. 
 \end{dfn}

\begin{thm}\label{LLLweakconv} The following hold true.
\begin{enumerate}[{\rm (i)}]
\item $\mathcal{L}$ is $\mathfrak{G}_{(n-1,1),(0,1)}$-convex .
\item $\mathcal{L}$ and $\mathcal{S}$ are weakly $\mathfrak{G}_{\lambda ,d}$-convex for any $\lambda$ and $d$.
\end{enumerate}
\end{thm}

 We follow the argument in Beukers~\cite{beukers} in which Theorem \ref{LLLweakconv} {\rm (i)} is proved.  In order to prove this theorem, we have several things to prepare.
 
 For simplicity, we write $I_1(\mathbf{t})$ for $I_{(1,\dots, 1)}(t_1,\dots, t_n)$. Let $g \in GL_n(\mathbb{R})$ and let
 $$\varPi : \mathbb{R}_{>0}^n \rightarrow \mathfrak{h}^n ; \mathbf{t} = (t_1, \cdots, t_n) \mapsto [g 
I_1(\mathbf{t}) ].$$
 We consider the condition for $\varPi (\mathbf{t})$ to be in $\mathcal{S}$, $\mathcal{L}$.
 Let 
 $$g I_1(\mathbf{t})= X(x_{ij}(\mathbf{t})) Y(r_i(\mathbf{t})) R(\mathbf{t})$$
  be the Iwasawa decomposition. Set 
  $$Q(\mathbf{t}) = (q_{ij}(\mathbf{t}))_{ij} := 
 g I_1(\mathbf{t}) {}^t\!I_1(\mathbf{t}) {}^t\!g =  g I_1(\mathbf{t}^2){}^t\!g .$$
Here ${}^t\! g$ denotes the transpose matrix of $g$. In the following we occasionally omit the parameter $\mathbf{t}$ for simplicity.
\begin{lem}\label{lemgcf1}
Define $\tilde{B}_{ij} := 
\left(
\begin{array}{lcl}
q_{ij}&\cdots&q_{in}\\
q_{j+1j}&\cdots&q_{j+1n}\\
&\vdots&\\
q_{nj}&\cdots&q_{nn}
\end{array}
\right)
$ for $i\leq j$, and set $B_{ij} := \det \tilde{B}_{ij}$. 

Define $\tilde{C_i}$ as the submatrix of $\tilde{B}_{i-1i-1}$ obtained by deletion of the second row and column (i.e. those crossing at $q_{ii}$), and set $C_i := \det \tilde{C}_i$.

Then the following equalities hold true.
\begin{enumerate}[{\rm (i)}]
\item $r_i^2 = B_{ii} / B_{i+1i+1}$.
\item $x_{ij} = B_{ij} / B_{jj}$.
\item $r_{i-1}^2 + x_{i-1i}^2 r_{i}^2 = C_{i} / B_{i+1i+1}$.
\end{enumerate}
Here $B_{n+1n+1} := 1$.
\end{lem}
\begin{proof}
See \cite[Theorem 3.6, Theorem 3.7]{beukers}.
\end{proof}

\begin{lem}\label{lemgcf2}
$B_{ij}$ and $C_i$ can be written in the form
$$B_{ij} (\mathbf{t}) = \sum_{\substack{I \subset \{1,\cdots, n\} \\ \# I = n-j+1 }} b_{ij,I} t_I^2,\ \ 
C_{i} (\mathbf{t}) = \sum_{\substack{I \subset \{1,\cdots, n\} \\ \# I = n-i+1 }} c_{i,I} t_I^2$$
where $b_{ij,I},\ c_{i,I} \in \mathbb{R}$ and $t_I = \prod_{k \in I} t_k$. In other words, if we set $\tau _k = t_k^2$, $B_{ij} (\mathbf{t})$ is a homogeneous polynomial of $\tau _k$ of degree $n-j+1$ which is at most degree one in each $\tau _k$.
\end{lem}
\begin{proof}

 Since the proofs are absolutely the same for $B_{ij}$ and $C_i$, we prove only for $B_{ij}$.
The claim that $B_{ij}(t)$ is a homogeneous of degree $n-j+1$ in $\tau _k$ holds since $q_{ij}$ are linear combinations of $\tau _k$.

 Any $\tilde{B}_{ij}$ can be written in the form
\begin{eqnarray*}
 A' g \left( \begin{array}{ccc} \tau _1&&\\ &\ddots&\\ &&\tau _n \end{array} \right) {}^t \! g B' 
				   = A \left( \begin{array}{ccc} \tau _1&&\\ &\ddots&\\ &&\tau _n \end{array} \right) B
\end{eqnarray*}
for some $A',A \in M_{n-j+1,n}(\mathbb{R})$ and $B',B \in M_{n,n-j+1}(\mathbb{R})$.
For any $k$, we can find $P,Q \in GL_{n-j+1}(\mathbb{R})$ such that there is at most one non-zero entry $1$ in the $k$ th column (resp. row) of $PA$ (resp.$BQ$).  Then $\tau _k$ appears in at most one place in $P\tilde{B}_{ij}Q$ which is at most degree one in $\tau _k$. Therefore, $B_{ij}=\det \tilde{B}_{ij} = \det P\tilde{B}_{ij}Q /(\det P \det Q)$ is at most degree one in $\tau _k$. This proves the lemma. 
\end{proof}

\begin{proof}[Proof of Theorem \ref{LLLweakconv}]

(i) By Lemma \ref{lemgcf1} and Lemma \ref{lemgcf2}, for any $\mathfrak{G}_{(n-1,1),(0,1)}$ geodesic $\varpi$, the set of $\tau = t^2 \in \mathbb{R}$ which satisfy $\varpi (t) \in \mathcal{L}$ is defined by $(n^2+n)/2-1$ linear inequalities in $\tau$, and hence an interval.

(ii) Similarly, for any $\mathfrak{G}_{\lambda,d}$ geodesic $\varpi$, the set of $\tau = t^2$ which satisfy $\varpi (t) \in \mathcal{L}$ is defined by $(n^2+n)/2-1$ inequalities of degree at most $\sum_{i} n_id_i$ in $\tau$, and hence a union of a finite number of intervals.
\end{proof}
\paragraph{Technical remarks}~\\
Suppose we take $\mathcal{L}=\mathcal{L}_{3/4}$ as a fundamental domain , and let $\varpi$ be a $\mathfrak{G}_{\lambda,d}$-geodesic.
From the argument in the proof of Theorem \ref{LLLweakconv}, we see the following.
\begin{cor}
We can practically compute $s_k,t_k$ in Algorithm \ref{GCF} by solving $(n^2+n)/2-1$ inequalities of degree at most $\sum_{i} n_id_i$ appeared in the proof of Theorem \ref{LLLweakconv}. \qed
\end{cor}
Therefore, in order to execute (forward loop of) Algorithm \ref{GCF}, it suffices to find 
\begin{enumerate}[(a)]
\item $B_0 \in \Gamma$ such that $\varpi (1) \in B_0\mathcal{L}$,  
\item $A_{k+1} \in \Gamma$ and $u_{k+1} > t_k$ such that $\varpi _k([t_k, u_{k+1}]) \subset A_{k+1} \mathcal{L}$ for each $k > 0$. 
\end{enumerate}
One method to do this is to use LLL-reduction, named after its inventors Lazlo Lovasz, Arjen Lenstra and Hendrik Lenstra \cite{LLL}. For any $z \in \mathfrak{h}^n$, LLL-reduction is literally the algorithm to find $A\in \Gamma$ such that $Az$ is LLL-reduced, i.e. $z \in A^{-1}\mathcal{L}$.
Therefore, we can find $B_0$ in (a). Furthermore, if we compute $\gamma_l \in \Gamma$ such that $\varpi_k(t_k+1/l) \in \gamma_l \mathcal{L}$ for $l \in \mathbb{N}_{>0}$, some  $t_k +1/l$ and $\gamma_l$ satisfies the condition of $u_k$ and $A_{k+1}$ in (b) by Proposition \ref{gcfwelldef}.
See, for example, \cite{LLL}, \cite{beukers} or \cite{cohen} for detailed explanation of LLL-reduction algorithm.

Another possibility to determine $A_{k+1} \in \Gamma$ in (b) is to compute the set $J = \{\gamma \in \Gamma \ |\  \gamma (\Sigma_{1/3,1})^{\circ} \cap (\Sigma_{1/3,1})^{\circ} \neq \emptyset \}$ which is a finite set by Proposition \ref{siegelsetFD}, in advance. Then we can always find $A_{k+1}$ in $J$.

In the numerical examples in Section \ref{sectex}, we choose the former method since this author cannot compute the set $J$ for $n\geq 3$ at present.

\subsection{Application to number fields}\label{sectapp}
 Now we present a generalization of Lagrange's theorem by applying the geodesic continued fraction algorithm (Algorithm \ref{GCF}) to Heegner objects. In this subsection we always assume $S=\{\infty\}$.
 
\paragraph{Number fields with rank one unit group}~\\
 We now return to the notations in Sections \ref{defheeg} to \ref{sectper}. 
 Let $F$ be a number field of degree $n$ such that the rank of its unit group is one, i.e. $F$ is either a real quadratic field, complex cubic field or totally imaginary quartic field.  In this case, there are exactly two archimedean embeddings $\sigma _i : F \rightarrow F_{\sigma_i} \subset \mathbb{C}$ $(i=1,2)$ (up to complex conjugates). Here $F_{\sigma_i}$ is the completion of $F$ with respect to $\sigma_i$. (We allow the choice: $\sigma_1$ is complex and $\sigma_2$ is real.)
 
 Let $w = {}^t\!(w _1 \cdots w _n)$ be a basis of $F$ over $\mathbb{Q}$, and let 
  $$\varpi : ( \mathbb{R}_{>0}^2)^1 \simeq \mathbb{R}_{>0} \rightarrow \mathfrak{h}^n ; t \mapsto [ w^{(1)}t\  w^{(2)}/t ] = [w^{(1)}t^2\  w^{(2)}]$$
be the ($\{\infty\}$-)Heegner object associated to $w$. Then by the change of variables $\tau = t^2$, we see that $\varpi$ is a $\mathfrak{G}_{(n_1,n_2),(1,0)}$-geodesic. Here $n_i =[F_{\sigma_i}:\mathbb{R}]$. In the following, we always take  the LLL-reduced set $\mathcal{L}=\mathcal{L}_{3/4}$ as a fundamental domain $\mathcal{F}$ which is weakly $\mathfrak{G}_{(n_1,n_2),(1,0)}$-convex by Theorem \ref{LLLweakconv}.

\begin{thm}[Rank one generalized Lagrange's theorem]\label{arch-main}~

 Let $\{(A_k, B_k, \varpi_k)\}_{k \in \mathbb{Z}}$ be the sequence obtained by Algorithm \ref{GCF} applied to $\varpi$. 
\begin{enumerate}[{\rm (i)}]
\item Then there exist $k_0 , k_1 \in \mathbb{N}, k_0<k_1$, and $\rho \in \mathbb{R}_{>1}$ such that $\varpi_{k_0}(t) = \varpi_{k_1}(\rho t)$ for any $t \in \mathbb{R}_{>0}$.
\item We have $B_{k_1}B_{k_0}^{-1} \in \Gamma_{\varpi}= \Gamma_{w}$ and $\epsilon := \varphi (B_{k_1}B_{k_0}^{-1}) \in \mathcal{O}_{w}^1$ is a non-torsion element satisfying $|\sigma_1(\epsilon)| =\rho$. In particular, $\epsilon$ generates a finite index subgroup of $\mathcal{O}_{w}^1$.
\end{enumerate}
\end{thm}
\begin{proof}
(i) This is clear from Proposition \ref{periodisunit}, Theorem \ref{perHeeg} and Theorem \ref{geodalgo}.

(ii) Since $\varpi_k = B_k^{-1}\varpi$, (i) implies $B_{k_1}B_{k_0}^{-1} \varpi(t)=\varpi (\rho t)$ $\forall t \in \mathbb{R}_{>0}$. Thus $B_{k_1}B_{k_0}^{-1} \in \Gamma_{\varpi}$. The rest follows from Proposition \ref{unitisperiod} and Corollary \ref{corunitperiod}.
\end{proof}
\begin{rmk}
It is unclear whether the minimal period of the algorithm always provides a fundamental unit, i.e. a generator of $\mathcal{O}_{w}^1/({\rm tor})$, because, now, the fundamental domain $\mathcal{F}$ is not a system of representatives.
\end{rmk}
\begin{rmk}
When we calculate $~\epsilon~$ in (ii), it is more convenient to consider $A_{k_0+1}\cdots A_{k_1} = B_{k_0}^{-1}B_{k_1} \in \Gamma_{B_{k_0}^{-1}\varpi}=B_{k_0}^{-1}\Gamma_{\varpi}B_{k_0}$. We easily verify $\epsilon:= \varphi_{w} (B_{k_1}B_{k_0}^{-1})= \varphi_{B_{k_0}^{-1}w}(B_{k_0}^{-1}B_{k_1})$.
\end{rmk}

\paragraph{Number fields with rank one $\chi$-component of unit group}~\\
Next, we extend the above result to $\chi$-component of Heegner objects. We keep the notations in Section \ref{sectchi}.
Let $F'$ be a number field with at most one complex archimedean place. We denote by $v_1',\dots, v_g'$ the archimedean places of $F'$, and suppose $v_2', \dots, v_g'$ are real. Let $F/F'$ be a quadratic extension such that $v_1'$ splits in $F/F'$ and $v_2', \dots , v_g'$ ramify in $F/F'$. We denote by $v_{11}, v_{12} $ the places of $F$ above $v_1'$, and by $v_i ~(i = 2, \dots ,g)$ the place of $F$ above $v_i'$. Set $d=[F':\mathbb{Q}]$ and $n=2d=[F:\mathbb{Q}]$.

 Let $w = {}^t\!(w _1 \cdots w _n)$ be a basis of $F$ over $\mathbb{Q}$, and let 
  $$\varpi : T^1:=( \mathbb{R}_{>0}^{g+1})^1  \rightarrow \mathfrak{h}^n; (t_{11},t_{12},t_2,\dots, t_g ) \mapsto [t_{11}v_{11}(w) ~ t_{12}v_{12}(w) ~ t_{2}v_{2}(w) ~  \cdots t_gv_g(w)] $$
be the ($\{\infty\}$-)Heegner object associated to $w$. Note that this definition does not depend on the choice of embeddings. Then, $(T^1)^{\chi} = \{(t,1/t,1, \dots, 1) \in T^1\ |\  t \in \mathbb{R}_{>0}\}$, and
\begin{eqnarray*}\varpi^{\chi} : (T^1)^{\chi} \simeq \mathbb{R}_{>0} \rightarrow \mathfrak{h}^n; t &\mapsto& [v_{11}(w)t ~ v_{12}(w)/t ~ v_{2}(w) ~  \cdots v_g(w)] \\
&& = [v_{11}(w)t^2 ~ v_{12}(w) ~ v_{2}(w)t ~  \cdots v_g(w)t]\end{eqnarray*}
is the $\chi$-component of the Heegner object. We see $\varpi^{\chi}$ is a $\mathfrak{G}_{(1,1,g-1),(2,0,1)}$-geodesic.
 
 \begin{thm}[$\chi$-rank one generalized Lagrange's theorem]\label{chiarch-main}~

 Let $\{(A_k, B_k, \varpi_k^{\chi})\}_{k \in \mathbb{Z}}$ be the sequence obtained by Algorithm \ref{GCF} applied to $\varpi^{\chi}$. 
\begin{enumerate}[{\rm (i)}]
\item Then there exist $k_0 , k_1 \in \mathbb{N}, k_0 < k_1$, and $\rho \in \mathbb{R}_{>1}$ such that $\varpi_{k_0}^{\chi}(t) = \varpi_{k_1}^{\chi}(\rho t)$ for any $t \in \mathbb{R}_{>0}$.
\item We have $B_{k_1}B_{k_0}^{-1} \in \Gamma_{\varpi}(\chi) \subset \Gamma_{w}$ and $\epsilon := \varphi (B_{k_1}B_{k_0}^{-1}) \in \mathcal{O}_{w}^1$ is a non-torsion element satisfying $|v_{11}(\epsilon)|=\rho$. In particular, $\epsilon$ generates a finite index subgroup of $(\mathcal{O}_{w}^1/({\rm tor}))^{\chi}$.
\end{enumerate}
\end{thm}
\begin{proof}
(i) This is clear from Proposition \ref{periodisunit}, Theorem \ref{chiperiodicity} and Theorem \ref{geodalgo}.

(ii) Since $\varpi_k = B_k^{-1}\varpi^{\chi}$, (i) implies $B_{k_1}B_{k_0}^{-1} \varpi^{\chi}(t)=\varpi^{\chi}(\rho t)$ $\forall t \in \mathbb{R}_{>0}$. Thus $B_{k_1}B_{k_0}^{-1} \in \Gamma_{\varpi}(\chi)$. The rest follows from Proposition \ref{unitisperiod}, Corollary \ref{corunitperiod} and Proposition \ref{chiunitperiod}.
\end{proof}

We present some numerical examples in Section \ref{sectex}.

\section{$\{\infty, p\}$-continued fractions}\label{p-gcf}
In this section, we extend the idea of geodesic continued fraction to $S$-ad\'elic setting. We restrict ourselves to the case where $n=2$ and $S=\{\infty,p\}$ for a prime number $p$ since, at present, this is the only case in which we can give satisfactory applications to number fields. In the following, we always assume $S=\{\infty,p\}$.

 Let $p$ be a prime number and let $\mathfrak{h}\times \mathfrak{h}_p^2$ be the generalized $\{\infty,p\}$-upper half space defined in Section \ref{defheeg}. Here $\mathfrak{h}=\{z \in \mathbb{C} \ |\  \Im(z) >0 \}$ is the usual Poincar\'e upper half plane identified with $\mathfrak{h}_{\infty}^2$ by
 $$\mathfrak{h} \simeq \mathfrak{h}_{\infty}^2; z=x+iy \mapsto 
 \left[
\begin{array}{cc}
  y&x       \\
  0&1    
\end{array}
\right] .$$
 
\subsection{Fundamental domain for the action of $SL_2(\mathbb{Z}[1/p])$ on $\mathfrak{h} \times \mathfrak{h}_p^2$}\label{sectp-FD}
 We recall the $p$-adic Iwasawa decomposition. 
\begin{thm}[Iwasawa decomposition for $GL_2(\mathbb{Q}_p)$]\label{p-iwa}~

For any $g \in GL_2(\mathbb{Q}_p)$, there exist $\lambda, \nu \in \mathbb{Z}, e \in \mathbb{Z}_{\geq 0} , u \in (\mathbb{Z}/p^e\mathbb{Z})^{\times}$, and $R \in GL_2(\mathbb{Z}_p)$ such that
$$g=p^{\lambda} 
\left(
\begin{array}{cc}
 p^{\nu} & \      \\
  \ &1    
\end{array}
\right)
\left(
\begin{array}{cc}
 p^{e} & \widetilde{u}      \\
  \ &1    
\end{array}
\right)R,
$$
where $\widetilde{u}$ is a lift of $u$ in $\mathbb{Z}_p$. Moreover, this presentation is unique up to the choice of $\widetilde{u}$.
\end{thm}
\begin{proof}
See \cite[Proposition 4.2.1]{goldfeld2}.
\end{proof}
\begin{rmk}
Note that any $z \in \mathfrak{h}_p^2 = GL_2(\mathbb{Q}_p)/\mathbb{Q}_p^{\times}GL_2(\mathbb{Z}_p)$ can be written uniquely in the form 
$z= \left[\left(
\begin{array}{cc}
 p^{\nu} & \      \\
  \ &1    
\end{array}
\right)
\left(
\begin{array}{cc}
 p^{e} & \widetilde{u}      \\
  \ &1    
\end{array}
\right)\right]=
\left[
\begin{array}{cc}
  p^{e+\nu}&p^{\nu}\widetilde{u}       \\
  \ &1    
\end{array}
\right] $
, with the same notations as in Theorem \ref{p-iwa}, up to the choice of $\widetilde{u}$. Here $[~]$ denotes the class in $\mathfrak{h}_p^2$ as in Section \ref{defheeg}.
\end{rmk}
\begin{lem}\label{lempope} With the notations above, the following equalities hold true in $\mathfrak{h}_p^2$.
\begin{enumerate}[{\rm (i)}]
\item $\left(
\begin{array}{cc}
  1&x     \\
  \ &1    
\end{array}
\right)
\left[
\begin{array}{cc}
  p^{e+\nu}&p^{\nu}\widetilde{u}       \\
  \ &1    
\end{array}
\right] =
\left[
\begin{array}{cc}
  p^{e+\nu}&p^{\nu}\widetilde{u}+x       \\
  \ &1    
\end{array}
\right]$ for $x \in \mathbb{Q}_p$.

\item $\left(
\begin{array}{cc}
  \ &-1       \\
  1 &\     
\end{array}
\right)
\left[
\begin{array}{cc}
  p^{e+\nu}&p^{\nu}\widetilde{u}       \\
  \ &1    
\end{array}
\right] =
\left[\left(
\begin{array}{cc}
  p^{-\nu}&\       \\
  \ &1    
\end{array}
\right)
\left(
\begin{array}{cc}
  p^{e}&-\widetilde{u^{-1}}       \\
  \ &1    
\end{array}
\right)
\right].$

\item $\left(
\begin{array}{cc}
  p^i&\        \\
  \ &p^{-i}     
\end{array}
\right)
\left[
\begin{array}{cc}
  p^{e+\nu}&p^{\nu}\widetilde{u}       \\
  \ &1    
\end{array}
\right] =
\left[\left(
\begin{array}{cc}
  p^{\nu +2i}&\       \\
  \ &1    
\end{array}
\right)
\left(
\begin{array}{cc}
  p^{e}&\widetilde{u}       \\
  \ &1    
\end{array}
\right)
\right] \text{ for } i \in \mathbb{Z}.$
\end{enumerate}
\end{lem}

\begin{proof}
(i) and (iii) are clear. We prove (ii). Let us fix a lift $\widetilde{u^{-1}} \in \mathbb{Z}_p$ of $u^{-1} \in (\mathbb{Z}/p^e\mathbb{Z})^{\times}$. Then we have $\widetilde{u}\widetilde{u^{-1}}+ap^e=1$ for some $a \in \mathbb{Z}_p$. Therefore, 

$\left(
\begin{array}{cc}
  \ &-1       \\
  1 &\     
\end{array}
\right)
\left(
\begin{array}{cc}
  p^{e+\nu}&p^{\nu}\widetilde{u}       \\
  \ &1    
\end{array}
\right) =
\left(
\begin{array}{cc}
  \ &-1   \\
  p^{e+\nu}&p^{\nu}\widetilde{u}        
\end{array}
\right) \equiv
\left(
\begin{array}{cc}
  \ &-1   \\
  p^{e+\nu}&p^{\nu}\widetilde{u}        
\end{array}
\right)
\left(
\begin{array}{cc}
 \widetilde{u}&a   \\
 -p^e&\widetilde{u^{-1}}      
\end{array}
\right) \\
=
\left(
\begin{array}{cc}
  p^e&-\widetilde{u^{-1}}   \\
  0&p^{\nu}(\widetilde{u}\widetilde{u^{-1}}+ap^e)        
\end{array}
\right) \equiv
\left(
\begin{array}{cc}
  p^{-\nu}&\       \\
  \ &1    
\end{array}
\right)
\left(
\begin{array}{cc}
  p^{e}&-\widetilde{u^{-1}}       \\
  \ &1    
\end{array}
\right)
$ (modulo $\mathbb{Q}_p^{\times}GL_2(\mathbb{Z}_p))$.
\end{proof}

\begin{prop}\label{p-FD}
 Let $\mathcal{D}=\mathcal{D}_0 := \{z \in \mathfrak{h}\  |\  |z|>1, -1/2 \leq \Re z < 1/2\} \cup \{e^{i \theta} \ |\   \theta \in [\pi /2, 2\pi /3] \}$, which is a fundamental domain for the action of $SL_2(\mathbb{Z})$ on $\mathfrak{h}$. 
Set $$\mathcal{D}_1:= 
\left(
\begin{array}{cc}
  1/p&\      \\
  \ & 1     
\end{array}
\right)
\mathcal{D}_0, \text{~}
 \mathcal{F}_0 := \mathcal{D}_0 \times \left\{
\left[
\begin{array}{cc}
  1&\       \\
  \ &1    
\end{array}
\right]
\right\},\text{ and }
\mathcal{F}_1:=\mathcal{D}_1 \times \left\{
\left[
\begin{array}{cc}
  1/p&\       \\
  \ &1    
\end{array}
\right]
\right\}. $$
Then $\mathcal{F}:=\mathcal{F}_0 \sqcup \mathcal{F}_1$ gives a fundamental domain for the action of $SL_2(\mathbb{Z}[1/p])$ on $\mathfrak{h} \times \mathfrak{h}_p^2$, that is, $\mathcal{F}$ satisfies the following:
\begin{enumerate}[{\rm (i)}]
\item tiling property: $$\mathfrak{h}\times \mathfrak{h}_p^2 = \bigcup_{\gamma \in SL_2(\mathbb{Z}[1/p])} \gamma \mathcal{F}.$$
\item finiteness: the set $\{\gamma \in SL_2(\mathbb{Z}[1/p])\ |\ \gamma \mathcal{F} \cap \mathcal{F} \neq \emptyset \}$ is a finite set.
\end{enumerate} 
\end{prop}

\begin{proof}
 By Theorem \ref{p-iwa} and Lemma \ref{lempope}, for $z= (z_{\infty}, z_p) \in \mathfrak{h} \times \mathfrak{h}_p^2$ there exists $\gamma _p \in SL_2(\mathbb{Z}[1/p])$ such that $$\gamma _p z_p \in \left\{
\left[
\begin{array}{cc}
 1 &  \     \\
 \  & 1  
\end{array}
\right], 
\left[
\begin{array}{cc}
  1/p& \      \\
  \ &  1    
\end{array}
\right]
 \right\}.$$
 
 First, suppose $\gamma _p z_p =\left[
\begin{array}{cc}
 1 &  \     \\
 \  & 1  
\end{array}
\right]$. Then there exists $\gamma _{\infty} \in SL_2(\mathbb{Z})$ such that $\gamma _{\infty} \gamma _p z_{\infty} \in \mathcal{D}_0$.  Since the stabilizer of $\left[
\begin{array}{cc}
 1 &  \     \\
 \  & 1  
\end{array}
\right]$ in $SL_2(\mathbb{Z}[1/p])$ is $SL_2(\mathbb{Z})$, we get $\gamma _{\infty} \gamma _p z \in \mathcal{F}_0$. 

 Next, suppose $\gamma _p z_p =\left[
\begin{array}{cc}
 1/p &  \     \\
 \  & 1  
\end{array}
\right]$. Then, taking into account that the stabilizer of $\left[
\begin{array}{cc}
 1/p &  \     \\
 \  & 1  
\end{array}
\right]$ is $\left(
\begin{array}{cc}
 1/p &  \     \\
 \  & 1  
\end{array}
\right) SL_2(\mathbb{Z}) \left(
\begin{array}{cc}
 p &  \     \\
 \  & 1  
\end{array}
\right)$ and $\mathcal{D}_1$ is a fundamental domain of this group, there exists $\gamma_{\infty} \in \left(
\begin{array}{cc}
 1/p &  \     \\
 \  & 1  
\end{array}
\right) SL_2(\mathbb{Z}) \left(
\begin{array}{cc}
 p &  \     \\
 \  & 1  
\end{array}
\right)$ such that $\gamma_{\infty} \gamma_p z_{\infty} \in \mathcal{D}_1$, and we get $\gamma _{\infty} \gamma _p z \in \mathcal{F}_1$. 

Furthermore, 
\begin{multline*}
\left\{\gamma \in SL_2(\mathbb{Z}[1/p])\  |\  \gamma \mathcal{F} \cap \mathcal{F} \neq \emptyset \right\} =\\
 \left\{\gamma \in SL_2(\mathbb{Z})\  |\  \gamma \mathcal{D}_0 \cap \mathcal{D}_0 \neq \emptyset \right\} \cup \left\{\gamma \in \left(
\begin{array}{cc}
 1/p &  \     \\
 \  & 1  
\end{array}
\right) SL_2(\mathbb{Z}) \left(
\begin{array}{cc}
 p &  \     \\
 \  & 1  
\end{array}
\right) \ \Big|\  \gamma \mathcal{D}_1 \cap \mathcal{D}_1 \neq \emptyset \right\}
\end{multline*}
is a finite set.
\end{proof}

\subsection{$\{\infty, p\}$-geodesic continued fractions}\label{sectp-gcf}

\begin{dfn}\label{p-geo}
We define a $p$-geodesic on $\mathfrak{h} \times \mathfrak{h}_p^2$ to be a map of the following form
$$\varpi : p^{\mathbb{Z}} \rightarrow \mathfrak{h} \times \mathfrak{h}_p^2 \ ;\  p^{\nu} \mapsto \left(z, 
\left[
\begin{array}{cc}
 a& p^{2\nu}b \     \\
 c& p^{2\nu}d  
\end{array}
\right] \right) $$
for $z \in \mathfrak{h}$ and $\left(
\begin{array}{cc}
 a& b \     \\
 c& d  
\end{array}
\right) \in GL_2(\mathbb{Q}_p)$.

\end{dfn}

 As in Section \ref{gcf} we consider the following algorithm. Set $\Gamma := SL_2(\mathbb{Z}[1/p])$.

 \begin{algo}[$\{\infty, p\}$-geodesic continued fraction algorithm]\label{p-gcfalgo}~
 
  Let $\varpi : p^{\mathbb{Z}} \rightarrow \mathfrak{h} \times \mathfrak{h}_p^2$ be a $p$-geodesic.  
 \begin{list}{}{}
 \item[{\it Preparation}:] Set $u_0=1$. Take $B_0 \in \Gamma$ such that $\varpi (1) \in B_0 \mathcal{F}$ and set $\varpi _0 = B_0{}^{-1} \varpi $.
 \item[{\it Forward loop}:] For $\varpi _k$ ($k\geq0$) we define $\varpi _{k+1}$ as follows.
 	\begin{itemize}
	
	\item Determine $A_{k+1} \in \Gamma$ such that $\varpi _k(p^{k+1}) \subset A_{k+1} \mathcal{F}$. 
	\item Define $\varpi _{k+1} := A_{k+1}{}^{-1} \varpi _k$, $B_{k+1}:= B_kA_{k+1}$.
	\end{itemize}
 \item[{\it Backward loop}:] For $\varpi _k$ ($k\leq0$) we define $\varpi _{k-1}$ as follows.
 	\begin{itemize}
	
	\item Determine $A_{k} \in \Gamma$ such that $\varpi _k(p^{k-1}) \subset A_{k}^{-1} \mathcal{F}$.
	\item Define $\varpi _{k-1} := A_{k} \varpi _k$, $B_{k-1}:= B_kA_{k}^{-1}$.
	\end{itemize} 
 \end{list}
 \end{algo}
 In the following, we refer to this algorithm simply as the geodesic algorithm.

The following theorem is a $p$-geodesic version of Theorem \ref{geodalgo}.
 
 \begin{thm}\label{p-gcfper}
  Let $\varpi : p^{\mathbb{Z}} \rightarrow \mathfrak{h}_{\infty}^2 \times \mathfrak{h}_p^2$ be a $p$-geodesic. Let $\nu_0 \in \mathbb{Z}$. Suppose $\varpi$ is periodic for $\nu\geq \nu_0$ under the action of $SL_2(\mathbb{Z}[1/p])$, i.e. there exist $\mu_0 \in \mathbb{N}_{>0}$ and $A \in SL_2(\mathbb{Z}[1/p])$ such that $A\varpi(p^{\nu}) = \varpi(p^{\nu+\mu_0})$ for all $\nu \geq \nu_0$. 
  
  Let $\{(A_k,B_k,\varpi _k)\}_k$ be the sequence taken by Algorithm \ref{p-gcfalgo}. Then there exist $M,N \in \mathbb{N}, M<N$, and $\mu \in \mathbb{N}_{>0}$ such that $\varpi_M(p^{\nu}) = \varpi_N(p^{\nu+\mu \mu_0})$ for all $\nu \geq \nu_0$. 
 \end{thm}
\begin{proof} For $k \in \mathbb{N}$, we have $B_{\mu_0k+\nu_0}^{-1}\varpi(p^{\mu_0k+\nu_0}) = B_{\mu_0k+\nu_0}^{-1} A^k \varpi(p^{\nu_0}) \in \mathcal{F}$, and hence $\varpi(p^{\nu_0}) \in A^{-k}B_{\mu_0k+\nu_0} \mathcal{F} \cap B_{\nu_0}\mathcal{F}$. On the other hand we have seen in Proposition \ref{p-FD} that $\{\gamma \in SL_2(\mathbb{Z}[1/p])\ |\linebreak \gamma \mathcal{F} \cap \mathcal{F} \neq \emptyset\}$ is a finite set. Therefore there exist $k,l \in \mathbb{N}, k<l$, such that $A^{-k}B_{\mu_0k+\nu_0}=A^{-l}B_{\mu_0l+\nu_0}$. 
  We get $\varpi_{\mu_0k+\nu_0}(p^{\nu}) = B_{\mu_0k+\nu_0}^{-1}\varpi(p^{\nu})=B_{\mu_0l+\nu_0}^{-1}A^{l-k}\varpi(p^{\nu}) = \varpi_{\mu_0l+\nu_0}(p^{\mu_0(l-k)+\nu})$ for $\nu \geq \nu_0$.
 \end{proof}

\subsection{$\{\infty ,p\}$-continued fractions}\label{sectp-cf}
 Next we give an explicit procedure of the above algorithm which is simple and similar to the classical continued fraction algorithm for real numbers. 
 
 We fix an isomorphism $\mathbb{C} \simeq \overline{\mathbb{Q}}_p$ of fields.
Let  $\mathfrak{h}$ be the Poincar\'e upper half plane as before, and let $\mathcal{D}$ be the fundamental domain for $SL_2(\mathbb{Z})$; $\mathcal{D}=\mathcal{D}_0 := \{z \in \mathfrak{h}\  |\  |z|>1, -1/2 \leq \Re z < 1/2\} \cup \{e^{i \theta} \ |\   \theta \in [\pi /2, 2\pi /3] \}$.
In the following, for $z \in \mathbb{C} = \overline{\mathbb{Q}}_p$, we denote by $\bar{z}$ the complex conjugate of $z$.

\begin{algo}[$\{\infty, p\}$-continued fraction algorithm]\label{p-cf}~

 For $z \in \mathfrak{h} \cap \mathbb{Q}_p \subset \mathbb{C} = \overline{\mathbb{Q}}_p$, we define its $\{\infty, p\}$-continued fraction expansion as follows.
 
 Set $x_1:= z$. For $x_k, k \geq1$, do the following.
\begin{list}{}{}
\item[{\it initialization}:]  If $x_k \in \mathbb{Z}_p$, set $z_k = x_k , \delta _k = +1$, otherwise, set $z_k =- 1/x_k \in \mathbb{Z}_p , \delta _k =-1$. 
\item[$p$-{\it adic reduction}:] Determine $a_{k} \in \{0, \dots , p^2-1 \}$ such that $z_k-a_{k} \equiv 0 \mod p^2$, and set $z_{k0} := (z_k-a_k)/p^2$. 
\item[$\infty$-{\it adic reduction}:] For $j \geq 0$, if $z_{kj} \in \mathcal{D}$, then set $x_{k+1} = z_{kj}$ and go to the initialization step.
If not, determine $b_{kj} \in \mathbb{Z}$ so that $-1/2 \leq \Re (z_{kj}-b_{kj}) < 1/2$ and set $z_{k j+1} = -1/(z_{kj}-b_{kj})$. 
\end{list}
Then $\infty$-adic reduction ends within a finite number of steps (See \cite[Lemma 2.3.1]{diamond}), and the algorithm is well defined.  We \underline{formally} express this expansion as
\begin{eqnarray*}
z &=&\delta _1( a_1 +p^2b_{10}- \cfrac{p^2}{b_{11}- \cfrac{1}{{\ddots}-\cfrac{1}{b_{1l_1}-\cfrac{1}{\delta _2(a_2+p^2b_{20}-\cfrac{p^2}{b_{21}-\cfrac{1}{\ddots}^{\ }})^{\delta _2 }}^{\ }}^{\ }}^{\ }}^{\ })^{\delta _1}\\
   &=& [\  \delta _1 ;\  a_1 ;\  b_{10}, \dots , b_{1l_1}  ;\   \delta _2 ; \ a_2 ;\   b_{20}, \dots , b_{2l_2}  ;  \dots ].
\end{eqnarray*}
Furthermore, we associate each step of the algorithm with an element in $SL_2(\mathbb{Z}[1/p])$ which corresponds to the linear fractional transformation of the step.  For $k\geq 1$, we define notations as follows.
\begin{list}{}{}
\item[{\it initialization}:] Set $D_k = \left(
\begin{array}{cc}
  1&    \\
  &1
\end{array}
\right)$ if $\delta _k =1$, $D_k = \left(
\begin{array}{cc}
  &-1    \\
  1&
\end{array}
\right)$ if $\delta _k=-1$.
\item[$p$-{\it adic reduction}:] Set $P_k=
 \left(
\begin{array}{cc}
  1/p&    \\
  &p
\end{array}
\right)
 \left(
\begin{array}{cc}
  1&-a_k    \\
  &1
\end{array}
\right)$.
\item[$\infty$-{\it adic reduction}:] Set $Q_k= \left(
\begin{array}{cc}
  0&-1    \\
  1&-b_{kl_k}
\end{array}
\right) \cdots 
\left(
\begin{array}{cc}
  0&-1    \\
  1&-b_{k0}
\end{array}
\right)$.
\end{list}
Moreover, we set $A_k^{-1}= Q_kP_kD_k$, $B_k=A_1\cdots A_k$. Then we have $z_k= D_k x_k$, $z_{k0} = P_k z_k$, $x_{k+1}=Q_k z_{k0}$, and $x_{k+1}= A_k^{-1}x_k = B_k^{-1}x_1$, where the action is a linear fractional transformation.
\end{algo}

 We claim the following theorem which is an analogue of Lagrange's theorem.
 
\begin{thm}\label{p-lag}
 If $z \in \mathfrak{h} \cap \mathbb{Q}_p \subset \mathbb{C} = \overline{\mathbb{Q}}_p$ is an imaginary quadratic irrational, then its $\{\infty, p\}$-continued fraction expansion becomes periodic. That is, there exist $k, l \in\mathbb{N}, k<l$, such that $x_k=x_l$.
\end{thm} 
We prove this theorem using the geodesic algorithm (Algorithm \ref{p-gcfalgo}) and the periodicity of Heegner objects (Theorem \ref{perHeeg}). 
First we relate Algorithm \ref{p-cf} to Algorithm \ref{p-gcfalgo}.

Let $z \in \mathfrak{h} \cap \mathbb{Q}_p \subset \mathbb{C} = \overline{\mathbb{Q}}_p$ and $z=[\  \delta _1 ;\  a_1 ;\  b_{10}, \dots , b_{1l_1}  ;\   \delta _2 ; \ a_2 ;\   b_{20}, \dots , b_{2l_2}  ;  \dots ]$ be its $\{\infty,p\}$-continued fraction expansion. We consider the following $p$-geodesic:
$$\varpi : p^{\mathbb{Z}} \rightarrow \mathfrak{h}\times \mathfrak{h}_p^2 \ ;\  p^{\nu} \mapsto \left(z, 
\left[
\begin{array}{cc}
  p^{2\nu}&z      \\
  0&1  
\end{array}
\right]\right).$$
We claim the following.

\begin{prop}\label{claim1} Let $z \in \mathfrak{h} \cap \mathbb{Z}_p \subset \mathbb{C} = \overline{\mathbb{Q}}_p$, and let the notations be as above.
\begin{enumerate}[{\rm (i)}]
\item $B_k^{-1}\varpi (p^k) \in \mathcal{F}_0= \mathcal{D} \times \left\{\left[
\begin{array}{cc}
  1&    \\
  &1
\end{array}
\right]\right\}$, for $k \geq 1$.
\item $D_{k+1}B_k^{-1}\varpi (p^{\nu}) = \left(z_{k+1}, 
\left[
\begin{array}{cc}
  p^{2\nu -2k}&z_{k+1}    \\
  0&1
\end{array}
\right]\right)$, for $\nu \geq k$.
\end{enumerate}
\end{prop}
We need the following lemma.

\begin{lem}\label{lemp2}
 Let $z \in \mathbb{Z}_p, 
 \left(
\begin{array}{cc}
  a&b      \\
  c&d 
\end{array}
\right) \in SL_2(\mathbb{Z})$ such that $\dfrac{az+b}{cz+d} \in \mathbb{Z}_p$. Then 

$$\left(
\begin{array}{cc}
  a&b     \\
  c&d
\end{array}
\right)
\left[
\begin{array}{cc}
  p^{\nu}&z      \\
  0&1
\end{array}
\right] =
\left[
\begin{array}{cc}
  p^{\nu}&\dfrac{az+b}{cz+d}    \\
  0&1  
\end{array}
\right]~\  {\it for}\  \nu \geq 0.$$
\end{lem}

\begin{proof} 
By the assumption: $z, \dfrac{az+b}{cz+d} \in \mathbb{Z}_p$, we have $\dfrac{1}{cz+d} = a -c \dfrac{az+b}{cz+d} \in \mathbb{Z}_p$. In particular, $cz+d \in \mathbb{Z}_p^{\times}$. Then, keeping in mind that all the equalities are modulo $\mathbb{Q}_p^{\times}GL_2(\mathbb{Z}_p)$, 

\begin{eqnarray*}
{\it LHS} &=& 
 \left[
\begin{array}{cc}
  a p^{\nu}& az+b    \\
  c p^{\nu}& cz+d  
\end{array}
\right]  =
\left[
\begin{array}{cc}
  \dfrac{a}{cz+d}p^{\nu}&\dfrac{az+b}{cz+d}    \\
  \dfrac{c}{cz+d}p^{\nu}&1  
\end{array}
\right] \\ &=&
\left[
\begin{array}{cc}
  \dfrac{a}{cz+d}p^{\nu} - \dfrac{az+b}{cz+d} \dfrac{c}{cz+d}p^{\nu}&\dfrac{az+b}{cz+d}    \\
  0&1  
\end{array}
\right]  \\ &=&
\left[
\begin{array}{cc}
  \dfrac{p^{\nu}}{(cz+d)^2}&\dfrac{az+b}{cz+d}    \\
  0&1  
\end{array}
\right] = {\it RHS}. \\[-14mm]
\end{eqnarray*}
\end{proof}
~

\begin{proof}[Proof of Proposition \ref{claim1}]
 We prove {\rm (i)} and {\rm (ii)} simultaneously by induction on $k \geq 1$. Note that $D_1=I$ (identity matrix) and $z=x_1=z_1$ since we have assumed $z \in \mathbb{Z}_p$.
\begin{list}{}{}
\item[$\underline{k=1}$:] It is easy to see that,
$$P_1\left[
\begin{array}{cc}
  p^2&x_1    \\
  &1
\end{array}
\right]=\left[
\begin{array}{cc}
  1&    \\
  &1
\end{array}
\right], \text{ and } 
Q_1\left[
\begin{array}{cc}
  1&    \\
  &1
\end{array}
\right]=\left[
\begin{array}{cc}
  1&    \\
  &1
\end{array}
\right],$$ 
in $\mathfrak{h}_p^2$. On the other hand, we have $A_1^{-1}x_1=x_2 \in \mathcal{D}$. This shows {\rm (i)}. 

 Now $z_{10}=(z_1 -a_1)/p^2 \in \mathbb{Z}_p$, $D_2Q_1 \in SL_2(\mathbb{Z})$ and $z_2 =D_2Q_1 z_{10} \in \mathbb{Z}_p$. Therefore, by Lemma \ref{lemp2}, we have $$D_2Q_1P_1D_1\varpi (p^{\nu}) =D_2Q_1 \left(z_{10}, \left[
\begin{array}{cc}
  p^{2\nu -2}&z_{10}    \\
  0&1
\end{array}
\right]\right) =
\left(z_2, \left[
\begin{array}{cc}
  p^{2\nu -2}&z_2    \\
  0&1
\end{array}
\right]\right)~~(\nu \geq 1).$$ Here we use Lemma \ref{lempope} in the first equality. This shows {\rm (ii)}.
\item[$\underline{k\geq 2}$:] By the identity (ii) of the induction hypothesis for $k-1$, we prove this case by applying the argument in the case $k=1$ to $z'=z_k$ and $\varpi'(p^{\nu})= D_kB_{k-1}\varpi(p^{\nu-k+1})$. This completes the proof. \qedhere

\end{list}
\end{proof}
The following proposition relates Algorithm \ref{p-cf} to Algorithm \ref{p-gcfalgo}. 
\begin{prop}\label{claim2}
Let $z \in \mathfrak{h} \cap \mathbb{Z}_p \subset \mathbb{C} = \overline{\mathbb{Q}}_p$, and let the notations be as above. Let $B_0 \in \Gamma=SL_2(\mathbb{Z}[1/p])$ such that $\varpi(1) \in B_0\mathcal{F}$, and let $A_k' =B_{k-1}^{-1}B_k \in \Gamma$ $(k \geq1)$, i.e. $A_1'= B_0^{-1}B_1$, $A_k'=A_k$ $(k\geq 2)$. Then we can take $B_0$ as a preparation step and $\{(A_k',B_k,\varpi_k)\}_{k \geq 1}$ as a forward loop of Algorithm \ref{p-gcfalgo} for $\varpi$, where $\varpi_k = B_k^{-1} \varpi$.
\end{prop}
\begin{proof}
This follows from Proposition \ref{claim1} (i) since Algorithm \ref{p-gcfalgo} is completely determined by $B_k$, and $B_k$ is characterized by the condition $\varpi(p^k) \in B_k \mathcal{F}$.
\end{proof}

Now, we prove Theorem \ref{p-lag}. Let $z \in \mathfrak{h} \cap \mathbb{Q}_p \subset \mathbb{C} = \overline{\mathbb{Q}}_p$ be an imaginary quadratic irrational. Then $F= \mathbb{Q}(z)$ is an imaginary quadratic field and $p$ splits completely in $F/\mathbb{Q}$. Let $z=[\  \delta _1 ;\  a_1 ;\  b_{10}, \dots , b_{1l_1}  ;\   \delta _2 ; \ a_2 ;\   b_{20}, \dots , b_{2l_2}  ;  \dots ]$ be the $\{\infty,p\}$-continued fraction expansion of $z$, and let $x_k,z_k,D_k,P_k$, etc. be as in Algorithm \ref{p-cf}. 
Let 
 $$\varpi_S:  p^{\mathbb{Z}} \rightarrow \mathfrak{h} \times \mathfrak{h}_p^2 \ ;\  p^{\nu} \mapsto \left(z, 
 \left[
\begin{array}{cc}
  z&p^{2\nu -e _0} \overline{z}     \\
  1&p^{2\nu -e _0} 
\end{array}
\right]\right),\   \text{ where } e _0 := v_p(z-\overline{z}),  $$
be the $S(=\{\infty,p\})$-Heegner object associated to the basis $w={}^t\!(z,1)$ of $F$ over $\mathbb{Q}$ with index $e_p=(0,-e_0)$. Let $\varpi$ be the $p$-geodesic as before:
$$\varpi : p^{\mathbb{Z}} \rightarrow \mathfrak{h}\times \mathfrak{h}_p^2 \ ;\  p^{\nu} \mapsto \left(z, 
\left[
\begin{array}{cc}
  p^{2\nu}&z      \\
  0&1  
\end{array}
\right]\right).$$

\begin{lem}\label{lemp1} Let $w _1, w _2 \in \mathbb{Q}_p \hookrightarrow \overline{\mathbb{Q}}_p = \mathbb{C}$. Suppose $w_1/w_2 \in \mathfrak{h}$, so that $w_1\overline{w_2}-\overline{w_1}w_2 \neq0$. Let $\alpha = v_p(w_2), \beta = v_p(\overline{w_2})$ be the additive $p$-adic valuations. The following hold true in $\mathfrak{h}_p^2$. 
\begin{enumerate}[{\rm (i)}]
\item $\left[
\begin{array}{cc}
  w_1&p^{\nu} \overline{w_1}      \\
  w_2&p^{\nu} \overline{w_2}  
\end{array}
\right] = 
\left[
\begin{array}{cc}
  w_1/w_2&p^{\nu -\alpha +\beta} \overline{w_1/w_2}      \\
  1				&p^{\nu -\alpha +\beta}  
\end{array}
\right]$~~ for $\nu \in \mathbb{Z}$.
\item If $\nu -\alpha +\beta \geq 0$, then

$\left[
\begin{array}{cc}
  w_1&p^{\nu} \overline{w_1}      \\
  w_2&p^{\nu} \overline{w_2}  
\end{array}
\right] = 
\left[
\begin{array}{cc}
 p^{\nu -\alpha +\beta +\gamma}&w_1/w_2      \\
 0			&1
\end{array}
\right]$, where $\gamma = v_p(w_1/w_2-\overline{w_1/w_2})$.
\end{enumerate}
\end{lem}
\begin{proof}
(i) Dividing the whole matrix on the left hand side by $w_2$ yields the right hand side. 

(ii) Subtracting the $p^{\nu -\alpha +\beta}$ multiple of the first column from the second column on the right hand side matrix in (i), and switching  the two columns yield the right hand side of (ii). 
\end{proof}

\begin{cor}\label{auxgeod}
We have $\varpi(p^{\nu})= \varpi_S(p^{\nu})$ for all $2\nu \geq e_0$. 
\end{cor}
\begin{proof} This is clear by Lemma \ref{lemp1} (ii). \end{proof}

\begin{proof}[Proof of Theorem \ref{p-lag}]

Let $z \in \mathfrak{h} \cap \mathbb{Q}_p \subset \mathbb{C} = \overline{\mathbb{Q}}_p$ be an imaginary quadratic irrational. Let $z=[\  \delta _1 ;\  a_1 ;\  b_{10}, \dots , b_{1l_1}  ;\   \delta _2 ; \ a_2 ;\   b_{20}, \dots , b_{2l_2}  ;  \dots ]$, $\varpi_S$, and $\varpi$ be as above. 
We may assume without loss of generality that $z \in \mathbb{Z}_p$ since the $\{\infty,p\}$-continued fraction expansions of $z$ and $-1/z$ differ only in $\delta_1$. 

Then, by Corollary \ref{auxgeod} and Theorem \ref{perHeeg}, the auxiliary $p$-geodesic $\varpi$ is periodic for $2\nu \geq e_0$, i.e. $\varpi$ satisfies the condition in Theorem \ref{p-gcfper}. Therefore, by Theorem \ref{p-gcfper} and Proposition \ref{claim2}, there exist $k,l \in \mathbb{N}, k<l$, and $\mu \in \mathbb{N}_{>0}$ such that $B_k^{-1} \varpi(p^{\nu}) = B_l^{-1}\varpi(p^{{\nu}+\mu})$ for $2\nu \geq e_0$. The archimedean part of this identity yields $x_{k+1}=x_{l+1}$.
This proves the theorem.
\end{proof}

Next we consider the period group $\Gamma_{\varpi_S}$.

\begin{prop}\label{minper} Let $z \in \mathfrak{h} \cap \mathbb{Z}_p \subset \mathbb{C} = \overline{\mathbb{Q}}_p$ be an imaginary quadratic irrational. We keep the notations above. Let $A \in \Gamma_{\varpi_S}$, that is, there exists $N\in \mathbb{Z}, N \neq 0$, such that $A\varpi_S (p^{\nu}) = \varpi_S (p^{\nu + N})$ for all $\nu \in \mathbb{Z}$. Suppose $N>0$. Then $x_{k+1}= x_{k+N+1}$ for $2k \geq e_0,~ k \geq 1$. 
Furthermore, $B_{k+N}^{-1}AB_k \in (\Gamma_{B_{k}^{-1}\varpi_S})_{{\rm tor}}$.
\end{prop}

\begin{proof}
For $2k \geq e _0$ and $k\geq 1$, we have $\varpi(p^k)=\varpi_S(p^k)$ by Corollary \ref{auxgeod}, and $$B_k^{-1} \varpi (p^k) , B_{k+N}^{-1} \varpi (p^{k+N}) \in \mathcal{D} \times 
\left\{\left[
\begin{array}{cc}
  1&    \\
  &1
\end{array}
\right]\right\}$$
by Proposition \ref{claim1}. On the other hand, $B_{k+N}^{-1} \varpi_S (p^{k+N}) = B_{k+N}^{-1}AB_{k} B_k^{-1} \varpi_S (p^k)$. Hence $C=B_{k+N}^{-1}AB_{k} \in SL_2(\mathbb{Z})$, and both $x_{k+1}$ and $x_{k+N+1}= Cx_{k+1}$ belong to $\mathcal{D}$. Therefore, $x_{k+1}$ and $x_{k+N+1}$ must be the same, and $C$ is a torsion element. Moreover, we see 
$C \in \Gamma_{B_{k}^{-1}w,S}= \Gamma_{B_k^{-1}\varpi_S},$
by Corollary \ref{corunitperiod} applied to the Heegner object $B_k^{-1}\varpi_S$ associated to the basis $B_{k}^{-1}w$ of $F$ over $\mathbb{Q}$, where $w={}^t\!(z,1)$. This completes the proof.
\end{proof}

\begin{cor}\label{pureper}
If $z \in \mathfrak{h} \cap \mathbb{Q}_p \subset \mathbb{C} = \overline{\mathbb{Q}}_p$ is an imaginary quadratic irrational such that $z = x_1 \in \mathcal{D} 
$ and $z_1-\overline{z_1} \in \mathbb{Z}_p^{\times}$, then its $\{\infty, p \}$-continued fraction is purely periodic. That is, there exists $N \in \mathbb{N}_{>0}$ such that $z=x_1=x_{N+1}$. Here $z_1$ is the element taken in the initialization step in the first loop of Algorithm \ref{p-cf}. 
\end{cor}

\begin{proof} 
Let $z'=p^2z_1$. Then $(k+1)$-th loop of $\{\infty , p\}$-continued fraction expansion of $z'$ coincides with $k$-th loop of $\{\infty , p\}$-continued fraction expansion of $z$. Proposition \ref{minper} applied to $z'$ finishes the proof.
\end{proof}

\subsection{$p$-units, the Pell-like equations and ideal classes}\label{sectp-app}
 In this section, we discuss $p$-units, the Pell-like equations and ideal classes of imaginary quadratic fields in terms of the $\{\infty, p\}$-continued fractions and its periods.
  
 Let $F \subset \mathbb{C}$ be an imaginary quadratic field and $p$ a prime number which splits completely in $F/ \mathbb{Q}$. We fix an isomorphism $\mathbb{C} \simeq \overline{\mathbb{Q}}_p$ of fields. Let $v$ be a place of $F$ which corresponds to the embedding $F \subset \mathbb{C} = \overline{\mathbb{Q}}_p$ and $\bar{v}$ its conjugate over $\mathbb{Q}$. We denote by $\mathcal{O}_F$ the ring of integers of $F$ and by $\mathfrak{p}$ (resp. $\bar{\mathfrak{p}}$) the prime ideal above $p$ which corresponds to the place $v$ (resp. $\bar{v}$). We have $(p) = \mathfrak{p} \bar{\mathfrak{p}}$. More precisely, let $d>0$ be a square free integer such that $F=\mathbb{Q}(\sqrt{-d})$. Let $\theta = \dfrac{-1+\sqrt{-d}}{2}$ if $-d \equiv 1\mod 4$, and $\theta = \sqrt{-d}$ otherwise, so that $\mathcal{O}_F = \mathbb{Z}[\theta]$. We see $\theta \in \mathbb{Z}_p$ by the assumption that $p$ splits completely in $F/\mathbb{Q}$.
 
 Combining the previous results, we get the following.
\begin{thm}\label{p-lag2} Consider the $\{\infty , p\}$-continued fraction expansion $$\theta = [ \delta _1 ;\  a_1 ;\  b_{10}, \dots , b_{1l_1}  ;\ \delta _2 ; \ a_2 ;\ b_{20}, \dots , b_{2l_2}  ;  \dots ].$$
We also take $x_k, z_k, A_k, B_k$, etc. as in Algorithm \ref{p-cf}.  
The following hold true.
\begin{enumerate}[{\rm (i)}]
\item The $\{\infty , p\}$-continued fraction expansion of $\theta$ is purely periodic.
\item Let $N$ be the smallest  positive integer such that $\theta = x_1=x_{N+1}$. Set $
\left(
\begin{array}{cc}
  q_N&r _N    \\
  s_N&t_N
\end{array}
\right):= B_N^{-1}\in SL_2(\mathbb{Z}[1/p])$. Then $\epsilon =s_N \theta + t_N \in F$ is a fundamental norm one $p$-unit of F, i.e. a generator of the rank one free abelian group $\mathcal{O}_F[1/p]^1/({\rm tor})$, where $\mathcal{O}_F[1/p]^1 =\{x \in \mathcal{O}_F[1/p]^{\times} \ |\  N_{F/\mathbb{Q}}(x)=1\}$.
\item In {\rm (ii)}, $s:=p^Ns_N , t:= p^Nt_N$ are relatively prime integers. Set $u:= p^N \epsilon =s\theta +t \in \mathcal{O}_F$. Then $u \mathcal{O}_F = \mathfrak{p}^{2N}$ as ideals. Furthermore, if there exists $x$ in $\mathcal{O}_F$ such that $u/x^2$ is a root of unity, then ${\rm ord} [\mathfrak{p}] = N$, and otherwise ${\rm ord} [\mathfrak{p}] = 2N$. Here ${\rm ord} [\mathfrak{p}]$ is the order of the ideal class $[\mathfrak{p}]$ in the ideal class group of $\mathcal{O}_F$.
\end{enumerate}
\end{thm}

\begin{proof}
(i) It is obvious that $\theta \in \mathcal{D}$. Since $p$ is unramified in $F$, $p \nmid d$ if $-d \equiv 1 \mod 4$ and $p \nmid 2d$ otherwise, we have $\theta-\overline{\theta} \in \mathbb{Z}_p^{\times}$. Therefore the conditions in Corollary \ref{pureper} are satisfied.

(ii) Let $w$ be the basis ${}^t\!(\theta,1)$ of $F$ over $\mathbb{Q}$. Let $\varpi_S$ be the $\{\infty,p\}$-Heegner object associated to $w$ with index $e=(0,0)$. First, since $(\theta,1)$ is a basis of $\mathcal{O}_F$ over $\mathbb{Z}$, by Proposition \ref{orderunits} (i), we see $\mathcal{O}_{w,S} = \mathcal{O}_F[1/p]$. 
Now $B_N^{-1}x_1=x_{1+N}=x_1$. Hence, by Corollary \ref{corunitperiod}, $B_N^{-1} \in \Gamma_{w,S}=\Gamma_{\varpi_S}\overset{\varphi}{\simeq} \mathcal{O}_F[1/p]^1$. Moreover, we easily see $\varphi(B_{N}^{-1})=s_N \theta + t_N=\epsilon$. 

It suffices to show $B_N^{-1}$ gives a generator of $\Gamma_{\varpi_S}/({\rm tor})$. Take $A \in \Gamma_{\varpi_S}$ which gives a generator of $\Gamma_{\varpi_S}/({\rm tor})$. Let $N' \in \mathbb{Z}, N' \neq 0$, such that $A \varpi_S(p^{\nu})=\varpi_S(p^{\nu+N'})$ for all $\nu \in \mathbb{Z}$. Suppose $N'>0$.
Then, by Proposition \ref{minper}, we get $x_2=x_{N'+2}$, and $B_{N'+1}^{-1}AB_{1} \in (\Gamma_{B_1^{-1}\varpi_S})_{\rm tor}$. On the other hand, since we have proved that the $\{\infty , p\}$-continued fraction expansion of $\theta$ is purely periodic, we have $x_i=x_{i+N'}$ for all $i\geq 1$. Indeed, since $x_2=x_{N'+2}$, we have  $x_i=x_{i+N'}$ for all $i\geq 2$, and since $x_1=x_{N+1}$, we have  $x_i=x_{i+N}$ for all $i\geq 1$. Thus we get $x_i = x_{i+N}=x_{i+N+N'}=x_{i+N'}$ for all $i \geq 1$. In particular, $x_{N'+1}=x_{1}$, and $A_{N'+1}=A_1=B_1$. 
Therefore, $B_{N'+1}=B_{N'}A_{N'+1}=B_{N'}B_1$, and $C:=B_{N'}^{-1}A \in (B_1\Gamma_{B_1^{-1}\varpi_S}B_1^{-1})_{\rm tor}=(\Gamma_{\varpi_S})_{\rm tor}$.  Hence, $B_{N'}$ also gives a generator of $\Gamma_{\varpi_S}/({\rm tor})$.

Now, by the minimality of $N$, we have $N|N'$ and $B_{N'}=B_N^r$ for some $r \in \mathbb{N}_{>0}$. Since $B_{N'}$ gives a generator of $\Gamma_{\varpi_S}/({\rm tor})$, we get $r=1$, and $B_N=B_{N'}$. This proves (ii). 

(iii) It is clear from the construction that $p^N q_N,p^{N} r_N,p^{N} s_N,p^{N} t_N \in \mathbb{Z}$. 
\begin{claim}
$v_p(\epsilon) =N$.
\end{claim}
\begin{proof}[Proof of Claim] 
Recall the notations $\psi$ and $L: \Gamma_{\varpi_S} \rightarrow T_S^1 \simeq p^{\mathbb{Z}} ; A' \mapsto |\varphi(A')|_p$ in Section \ref{sectper}. By the proof of (ii), $B_N^{-1}A \in \Gamma_{\varpi_S}$ is a torsion element. Therefore, $|\epsilon|_p=L(B_N^{-1})=L(A^{-1})$. On the  other hand, by Proposition \ref{unitisperiod}, $L(A^{-1})= \psi(A^{-1})=p^{-N}$. Thus $v_p(\epsilon)=N$. 
\end{proof}
 Therefore, $v_p(s\theta +t)=p^{2N}$. On the other hand, we have $(s\theta +t) (\overline{s\theta +t}) =N_{F/\mathbb{Q}}(s\theta +t) =p^{2N}$ since $N_{F/\mathbb{Q}}(\epsilon)=1$. Thus, we see $v_p(s \overline{\theta}+t)=0$ and $s,t$ are relatively prime.  It also follows that $u \mathcal{O}_F = \mathfrak{p}^{2N}$.
 
 For the latter half of the statement, let $l= {\rm ord}[\mathfrak{p}]$ and $x \mathcal{O}_F = \mathfrak{p}^l$ for $x \in \mathcal{O}_F$. Then $x^2/p^l$ is a norm one $p$-unit. Since we already know in {\rm (ii)} that $\epsilon = u/p^N$ is a fundamental norm one $p$-unit, there exist $n \in \mathbb{Z}$ and a root of unity $\zeta \in F$ such that $x^2/p^l=\left(u/p^N\right)^n \times \zeta$. In particular, we have $l=v_p(x^2/p^l)=v_p((u/p^N)^n)=nN$. Now, since $u \mathcal{O}_F = \mathfrak{p}^{2N}$, we have $l | 2N$, and since $N,l>0$, we have $n>0$.
 Therefore, there are only two cases: $n=1$ or $n=2$. If $n=1$, then ${\rm ord}[\mathfrak{p}]=N$, and if $n=2$, then ${\rm ord}[\mathfrak{p}]=2N$. This proves the theorem.
\end{proof}

\begin{cor} We keep the notations in Theorem \ref{p-lag2}.
\begin{enumerate}[{\rm (i)}]
\item If $-d \not\equiv 1 \mod 4$, then the solutions $(x,y,\nu)$ to the Pell-like equation $x^2+d y^2=p^{2\nu},$ where $x$ and $y$ are relatively prime integers, and $\nu \in \mathbb{N}$, are exactly of the form $$x+ \sqrt{-d} y = \zeta (t \pm \sqrt{-d} s)^r, \nu=rN,$$ where $\zeta $ is a root of unity in $F$ and $ r \in \mathbb{N}$.
\item If $-d \equiv 1 \mod 4$, then then the solutions $(x,y,\nu)$ to the Pell-like equation $4x^2-4xy+(1+d)y^2=4p^{2\nu},$ where $x$ and $y$ are relatively prime integers, and $\nu \in \mathbb{N}$, are exactly of the form $$x+ \dfrac{-1+ \sqrt{-d}}{2} y = \zeta (t+ \dfrac{-1\pm \sqrt{-d}}{2} s)^r, \nu=rN,$$ where $\zeta $ is a root of unity in $F$ and $ r \in \mathbb{N}$.
\end{enumerate}
\end{cor}
\begin{proof}
Let $\theta$ be as above, so that $\mathcal{O}_F=\mathbb{Z}+\mathbb{Z} \theta$. For relatively prime integers $x,y$ and $\nu \in \mathbb{N}$,
\begin{eqnarray*}
N_{F/\mathbb{Q}}(x+y\theta) = p^{2\nu} &\Leftrightarrow& (x+y\theta) = \mathfrak{p}^{2\nu} \text{ or } \bar{\mathfrak{p}}^{2\nu}
\Leftrightarrow \left(\dfrac{x+y\theta}{p^{\nu}} \right) = \mathfrak{p}^{\nu} \bar{\mathfrak{p}}^{-\nu} \text{ or } \mathfrak{p}^{-\nu}\bar{\mathfrak{p}}^{\nu}\\
&\Leftrightarrow& N|\nu \text{ and } \left(\dfrac{x+y\theta}{p^{\nu}}\right)= \left(\dfrac{u}{p^N}\right)^{\nu/N}
\end{eqnarray*}
where, for $\alpha \in \mathcal{O}_F$, $(\alpha)$ denotes the integral ideal generated by $\alpha$. This shows the corollary.
\end{proof}

\clearpage
\section{Examples}\label{sectex}
In this section we present some numerical examples. 
\subsection{Geodesic continued fractions}
Here we treat the case where $S=\{\infty\}$. We use Mathematica for calculation.
\paragraph{Real quadratic field}~

 Let $d>1$ be a positive square free integer. Let $F$ be the real quadratic field $ \mathbb{Q}(\sqrt{d})$. We consider the Heegner object associated to $w ={}^t\!(\sqrt{d}, 1)$, that is, 
 $$\varpi : \mathbb{R}_{>0} \rightarrow \mathfrak{h}\  ;\  t \mapsto 
\left[
\begin{array}{cc}
 -\sqrt{d} & \sqrt{d} t \\
 1 & t
\end{array}
\right] .$$ 
 In the following, we give some examples of the sequence $\{A_k,B_k,[B_k^{-1}w]\}_k$ and $\{s_k , t_k\}_k$ obtained by Algorithm 1 (forward loop) applied to $\varpi(t)$. Here $[v] = v_1/v_2$ for $v={}^t\!(v_1, v_2) \in \mathbb{R}^2$. Note that by Corollary \ref{corunitperiod}, $\varpi_k(t)=\varpi_l(\rho t)$ holds for some $\rho \in \mathbb{R}_{>0}$ if and only if $[B_k^{-1}w] = [B_l^{-1}w]$. We write approximate values for $s_k,t_k$.

\begin{multicols}{2}
{\footnotesize
\noindent\underline{$d=2$}

\noindent\(\left\{
preparation; \left(
\begin{array}{cc}
 1 & 0 \\
 0 & 1
\end{array}
\right),\sqrt{2}\right\}\) :~period start

\noindent\(\{0.69108,1.44701\}\)

\noindent\(\left\{\left(
\begin{array}{cc}
 1 & 1 \\
 0 & 1
\end{array}
\right),\left(
\begin{array}{cc}
 1 & 1 \\
 0 & 1
\end{array}
\right),-1+\sqrt{2}\right\}\)

\noindent\(\{1.44701,2.96306\}\)

\noindent\(\left\{\left(
\begin{array}{cc}
 0 & -1 \\
 1 & 0
\end{array}
\right),\left(
\begin{array}{cc}
 1 & -1 \\
 1 & 0
\end{array}
\right),-1-\sqrt{2}\right\}\)

\noindent\(\{1.96703,4.02791\}\)

\noindent\(\left\{\left(
\begin{array}{cc}
 1 & -1 \\
 0 & 1
\end{array}
\right),\left(
\begin{array}{cc}
 1 & -2 \\
 1 & -1
\end{array}
\right),-\sqrt{2}\right\}\)

\noindent\(\{4.02791,8.43379\}\)

\noindent\(\left\{\left(
\begin{array}{cc}
 1 & -1 \\
 0 & 1
\end{array}
\right),\left(
\begin{array}{cc}
 1 & -3 \\
 1 & -2
\end{array}
\right),1-\sqrt{2}\right\}\)

\noindent\(\{8.43379,17.27\}\)

\noindent\(\left\{\left(
\begin{array}{cc}
 0 & -1 \\
 1 & 0
\end{array}
\right),\left(
\begin{array}{cc}
 -3 & -1 \\
 -2 & -1
\end{array}
\right),1+\sqrt{2}\right\}\)

\noindent\(\{11.4647,23.4764\}\)

\noindent\(\left\{\left(
\begin{array}{cc}
 1 & 1 \\
 0 & 1
\end{array}
\right),\left(
\begin{array}{cc}
 -3 & -4 \\
 -2 & -3
\end{array}
\right),\sqrt{2}\right\}\) :~period start

\noindent\(\{23.4764,49.1557\}\)

\noindent\(\left\{\left(
\begin{array}{cc}
 1 & 1 \\
 0 & 1
\end{array}
\right),\left(
\begin{array}{cc}
 -3 & -7 \\
 -2 & -5
\end{array}
\right),-1+\sqrt{2}\right\}\)

\noindent\(\{49.1557,100.657\}\)

\noindent\(\left\{\left(
\begin{array}{cc}
 0 & -1 \\
 1 & 0
\end{array}
\right),\left(
\begin{array}{cc}
 -7 & 3 \\
 -5 & 2
\end{array}
\right),-1-\sqrt{2}\right\}\)

\noindent\(\{66.8211,136.83\}\)

\noindent\(\left\{\left(
\begin{array}{cc}
 1 & -1 \\
 0 & 1
\end{array}
\right),\left(
\begin{array}{cc}
 -7 & 10 \\
 -5 & 7
\end{array}
\right),-\sqrt{2}\right\}\)

~\\
\noindent
$u=\varphi (
\left(
\begin{array}{cc}
 -3 & -4 \\
 -2 & -3
\end{array}
\right)
) = -2\sqrt{2}-3$ is a (fundamental) norm one unit of $\mathbb{Z}[\sqrt{2}]$.

\noindent\underline{$d=3$}

\noindent\(\left\{
preparation; \left(
\begin{array}{cc}
 1 & 0 \\
 0 & 1
\end{array}
\right),\sqrt{3}\right\}\) :~period start

\noindent\(\{0.742955,1.34598\}\)

\noindent\(\left\{\left(
\begin{array}{cc}
 1 & 1 \\
 0 & 1
\end{array}
\right),\left(
\begin{array}{cc}
 1 & 1 \\
 0 & 1
\end{array}
\right),-1+\sqrt{3}\right\}\)

\noindent\(\{1.34598,3.73205\}\)

\noindent\(\left\{\left(
\begin{array}{cc}
 1 & 1 \\
 0 & 1
\end{array}
\right),\left(
\begin{array}{cc}
 1 & 2 \\
 0 & 1
\end{array}
\right),-2+\sqrt{3}\right\}\)

\noindent\(\{3.73205,4.40807\}\)

\noindent\(\left\{\left(
\begin{array}{cc}
 0 & -1 \\
 1 & 1
\end{array}
\right),\left(
\begin{array}{cc}
 2 & 1 \\
 1 & 1
\end{array}
\right),1+\sqrt{3}\right\}\)

\noindent\(\{3.73205,10.348\}\)

\noindent\(\left\{\left(
\begin{array}{cc}
 1 & 1 \\
 0 & 1
\end{array}
\right),\left(
\begin{array}{cc}
 2 & 3 \\
 1 & 2
\end{array}
\right),\sqrt{3}\right\}\) :~period start

\noindent\(\{10.348,18.747\}\)

\noindent\(\left\{\left(
\begin{array}{cc}
 1 & 1 \\
 0 & 1
\end{array}
\right),\left(
\begin{array}{cc}
 2 & 5 \\
 1 & 3
\end{array}
\right),-1+\sqrt{3}\right\}\)

\noindent\(\{18.747,51.9808\}\)

\noindent\(\left\{\left(
\begin{array}{cc}
 1 & 1 \\
 0 & 1
\end{array}
\right),\left(
\begin{array}{cc}
 2 & 7 \\
 1 & 4
\end{array}
\right),-2+\sqrt{3}\right\}\)

\noindent\(\{51.9808,61.3965\}\)

\noindent\(\left\{\left(
\begin{array}{cc}
 0 & -1 \\
 1 & 1
\end{array}
\right),\left(
\begin{array}{cc}
 7 & 5 \\
 4 & 3
\end{array}
\right),1+\sqrt{3}\right\}\)

\noindent\(\{51.9808,144.129\}\)

\noindent\(\left\{\left(
\begin{array}{cc}
 1 & 1 \\
 0 & 1
\end{array}
\right),\left(
\begin{array}{cc}
 7 & 12 \\
 4 & 7
\end{array}
\right),\sqrt{3}\right\}\) :~period start

\noindent\(\{144.129,261.113\}\)

\noindent\(\left\{\left(
\begin{array}{cc}
 1 & 1 \\
 0 & 1
\end{array}
\right),\left(
\begin{array}{cc}
 7 & 19 \\
 4 & 11
\end{array}
\right),-1+\sqrt{3}\right\}\)

~\\
\noindent
$u=\varphi (
\left(
\begin{array}{cc}
 2 & 3 \\
 1 & 2
\end{array}
\right)
) = \sqrt{3}+2$ is a (fundamental) norm one unit of $\mathbb{Z}[\sqrt{3}]$.

}
\end{multicols}
{\small
Here a fundamental norm one unit of $\mathcal{O}$ means a generator of $\mathcal{O}^{1}/(\text{torsion})$, where $\mathcal{O}^1 \subset \mathcal{O}^{\times}$ is the subgroup of norm one units.
}
\clearpage

\paragraph{Complex cubic field $\mathbb{Q}(d^{1/3})$}~

 Let $d>1$ be a positive cubic free integer. Let $F$ be the complex cubic field $\mathbb{Q}(d^{1/3})$. We consider the Heegner object associated to $w ={}^t\!(d^{2/3}, d^{1/3}, 1)$, that is, 
 $$\varpi : \mathbb{R}_{>0} \rightarrow \mathfrak{h}^3\  ;\  t \mapsto 
\left[
\begin{array}{ccc}
 -\frac{1}{2} \sqrt{3} d^{2/3} & -\frac{d^{2/3}}{2} & d^{2/3} t \\
 \frac{1}{2} \sqrt{3} d^{1/3} & -\frac{d^{1/3}}{2} & d^{1/3} t \\
 0 & 1 & t
\end{array}
\right].$$ 
 In the following, we give some examples of the sequence $\{A_k,B_k,[B_k^{-1}w]\}_k$ and $\{s_k , t_k\}_k$ obtained by Algorithm 1 (forward loop) applied to $\varpi(t)$. Here $[v] = (v_1/v_3,v_2/v_3)$ for $v={}^t\!(v_1, v_2, v_3) \in \mathbb{R}^3$. Note that by Corollary \ref{corunitperiod}, $\varpi_k(t)=\varpi_l(\rho t)$ holds for some $\rho \in \mathbb{R}_{>0}$ if and only if $[B_k^{-1}w] = [B_l^{-1}w]$. We write approximate values for $s_k,t_k$. Set $\theta := d^{1/3}$.

\begin{multicols}{2}
{\tiny
\noindent\underline{$d=2$}

\noindent\(\left\{
preparation; \left(
\begin{array}{ccc}
 1 & 0 & 0 \\
 0 & 1 & 0 \\
 0 & 0 & 1
\end{array}
\right),\theta ^2,\theta \right\}\) :~period start

\noindent\(\{0.410037,1.09074\}\)

\noindent\(\left\{\left(
\begin{array}{ccc}
 1 & 0 & 1 \\
 0 & 1 & 0 \\
 0 & 0 & 1
\end{array}
\right),\left(
\begin{array}{ccc}
 1 & 0 & 1 \\
 0 & 1 & 0 \\
 0 & 0 & 1
\end{array}
\right),-1+\theta ^2,\theta \right\}\)

\noindent\(\{1.09074,1.2194\}\)

\noindent\(\left\{\left(
\begin{array}{ccc}
 1 & 0 & 0 \\
 0 & 1 & 1 \\
 0 & 0 & 1
\end{array}
\right),\left(
\begin{array}{ccc}
 1 & 0 & 1 \\
 0 & 1 & 1 \\
 0 & 0 & 1
\end{array}
\right),-1+\theta ^2,-1+\theta \right\}\)

\noindent\(\{1.2194,1.83997\}\)

\noindent\(\left\{\left(
\begin{array}{ccc}
 1 & 1 & 0 \\
 0 & 1 & 0 \\
 0 & 0 & 1
\end{array}
\right),\left(
\begin{array}{ccc}
 1 & 1 & 1 \\
 0 & 1 & 1 \\
 0 & 0 & 1
\end{array}
\right),-\theta +\theta ^2,-1+\theta \right\}\)

\noindent\(\{1.83997,2.1304\}\)

\noindent\(\left\{\left(
\begin{array}{ccc}
 -1 & 0 & 0 \\
 0 & 0 & -1 \\
 0 & -1 & 0
\end{array}
\right),\left(
\begin{array}{ccc}
 -1 & -1 & -1 \\
 0 & -1 & -1 \\
 0 & -1 & 0
\end{array}
\right),\theta ,1+\theta +\theta ^2\right\}\)

\noindent\(\{2.05068,2.67734\}\)

\noindent\(\left\{\left(
\begin{array}{ccc}
 0 & -1 & 0 \\
 -1 & 0 & 0 \\
 0 & 0 & -1
\end{array}
\right),\left(
\begin{array}{ccc}
 1 & 1 & 1 \\
 1 & 0 & 1 \\
 1 & 0 & 0
\end{array}
\right),1+\theta +\theta ^2,\theta \right\}\)

\noindent\(\{2.05068,3.09429\}\)

\noindent\(\left\{\left(
\begin{array}{ccc}
 1 & 1 & 1 \\
 0 & 1 & 0 \\
 0 & 0 & 1
\end{array}
\right),\left(
\begin{array}{ccc}
 1 & 2 & 2 \\
 1 & 1 & 2 \\
 1 & 1 & 1
\end{array}
\right),\theta ^2,\theta \right\}\) :~period start

\noindent\(\{3.09429,8.23114\}\)

\noindent\(\left\{\left(
\begin{array}{ccc}
 1 & 0 & 1 \\
 0 & 1 & 0 \\
 0 & 0 & 1
\end{array}
\right),\left(
\begin{array}{ccc}
 1 & 2 & 3 \\
 1 & 1 & 3 \\
 1 & 1 & 2
\end{array}
\right),-1+\theta ^2,\theta \right\}\)

\noindent\(\{8.23114,9.20206\}\)

\noindent\(\left\{\left(
\begin{array}{ccc}
 1 & 0 & 0 \\
 0 & 1 & 1 \\
 0 & 0 & 1
\end{array}
\right),\left(
\begin{array}{ccc}
 1 & 2 & 5 \\
 1 & 1 & 4 \\
 1 & 1 & 3
\end{array}
\right),-1+\theta ^2,-1+\theta \right\}\)

\noindent\(\{9.20206,13.8851\}\)

\noindent\(\left\{\left(
\begin{array}{ccc}
 1 & 1 & 0 \\
 0 & 1 & 0 \\
 0 & 0 & 1
\end{array}
\right),\left(
\begin{array}{ccc}
 1 & 3 & 5 \\
 1 & 2 & 4 \\
 1 & 2 & 3
\end{array}
\right),-\theta +\theta ^2,-1+\theta \right\}\)

\noindent\(\{13.8851,16.0768\}\)

\noindent\(\left\{\left(
\begin{array}{ccc}
 -1 & 0 & 0 \\
 0 & 0 & -1 \\
 0 & -1 & 0
\end{array}
\right),\left(
\begin{array}{ccc}
 -1 & -5 & -3 \\
 -1 & -4 & -2 \\
 -1 & -3 & -2
\end{array}
\right),\theta ,1+\theta +\theta ^2\right\}\)

\noindent\(\{15.4751,20.2042\}\)

\noindent\(\left\{\left(
\begin{array}{ccc}
 0 & -1 & 0 \\
 -1 & 0 & 0 \\
 0 & 0 & -1
\end{array}
\right),\left(
\begin{array}{ccc}
 5 & 1 & 3 \\
 4 & 1 & 2 \\
 3 & 1 & 2
\end{array}
\right),1+\theta +\theta ^2,\theta \right\}\)

\noindent\(\{15.4751,23.3506\}\)

\noindent\(\left\{\left(
\begin{array}{ccc}
 1 & 1 & 1 \\
 0 & 1 & 0 \\
 0 & 0 & 1
\end{array}
\right),\left(
\begin{array}{ccc}
 5 & 6 & 8 \\
 4 & 5 & 6 \\
 3 & 4 & 5
\end{array}
\right),\theta ^2,\theta \right\}\) :~period start

\noindent\(\{23.3506,62.1152\}\)

\noindent\(\left\{\left(
\begin{array}{ccc}
 1 & 0 & 1 \\
 0 & 1 & 0 \\
 0 & 0 & 1
\end{array}
\right),\left(
\begin{array}{ccc}
 5 & 6 & 13 \\
 4 & 5 & 10 \\
 3 & 4 & 8
\end{array}
\right),-1+\theta ^2,\theta \right\}\)

\noindent\(\{62.1152,69.4421\}\)

\noindent\(\left\{\left(
\begin{array}{ccc}
 1 & 0 & 0 \\
 0 & 1 & 1 \\
 0 & 0 & 1
\end{array}
\right),\left(
\begin{array}{ccc}
 5 & 6 & 19 \\
 4 & 5 & 15 \\
 3 & 4 & 12
\end{array}
\right),-1+\theta ^2,-1+\theta \right\}\)

\noindent\(\{69.4421,104.782\}\)

\noindent\(\left\{\left(
\begin{array}{ccc}
 1 & 1 & 0 \\
 0 & 1 & 0 \\
 0 & 0 & 1
\end{array}
\right),\left(
\begin{array}{ccc}
 5 & 11 & 19 \\
 4 & 9 & 15 \\
 3 & 7 & 12
\end{array}
\right),-\theta +\theta ^2,-1+\theta \right\}\)

~\\
\noindent
$u=\varphi (
\left(
\begin{array}{ccc}
 1 & 2 & 2 \\
 1 & 1 & 2 \\
 1 & 1 & 1
\end{array}
\right)
) =1+\theta+\theta ^2$ is a (fundamental) norm one unit of $\mathbb{Z}[2^{1/3}]$.

\noindent\underline{$d=3$}

\noindent\(\left\{preparation; \left(
\begin{array}{ccc}
 1 & 0 & 1 \\
 0 & 1 & 0 \\
 0 & 0 & 1
\end{array}
\right),-1+\theta ^2,\theta \right\}\) :~period start

\noindent\(\{0.987248,1.13841\}\)

\noindent\(\left\{\left(
\begin{array}{ccc}
 1 & 0 & 0 \\
 0 & 1 & 1 \\
 0 & 0 & 1
\end{array}
\right),\left(
\begin{array}{ccc}
 1 & 0 & 1 \\
 0 & 1 & 1 \\
 0 & 0 & 1
\end{array}
\right),-1+\theta ^2,-1+\theta \right\}\)

\noindent\(\{1.13841,1.48185\}\)

\noindent\(\left\{\left(
\begin{array}{ccc}
 1 & 1 & 0 \\
 0 & 1 & 0 \\
 0 & 0 & 1
\end{array}
\right),\left(
\begin{array}{ccc}
 1 & 1 & 1 \\
 0 & 1 & 1 \\
 0 & 0 & 1
\end{array}
\right),-\theta +\theta ^2,-1+\theta \right\}\)

\noindent\(\{1.48185,2.43748\}\)

\noindent\(\left\{\left(
\begin{array}{ccc}
 1 & 0 & 1 \\
 0 & 1 & 0 \\
 0 & 0 & 1
\end{array}
\right),\left(
\begin{array}{ccc}
 1 & 1 & 2 \\
 0 & 1 & 1 \\
 0 & 0 & 1
\end{array}
\right),-1-\theta +\theta ^2,-1+\theta \right\}\)

\noindent\(\{2.43748,2.60848\}\)

\noindent\(\left\{\left(
\begin{array}{ccc}
 -1 & 0 & 0 \\
 0 & 0 & -1 \\
 0 & -1 & 0
\end{array}
\right),\left(
\begin{array}{ccc}
 -1 & -2 & -1 \\
 0 & -1 & -1 \\
 0 & -1 & 0
\end{array}
\right),-\frac{1}{2}+\frac{\theta }{2}-\frac{\theta ^2}{2},\frac{1}{2}+\frac{\theta }{2}+\frac{\theta ^2}{2}\right\}\)

\noindent\(\{2.19262,3.40016\}\)

\noindent\(\left\{\left(
\begin{array}{ccc}
 1 & 0 & -1 \\
 0 & 1 & 1 \\
 0 & 0 & 1
\end{array}
\right),\left(
\begin{array}{ccc}
 -1 & -2 & -2 \\
 0 & -1 & -2 \\
 0 & -1 & -1
\end{array}
\right),\frac{1}{2}+\frac{\theta }{2}-\frac{\theta ^2}{2},-\frac{1}{2}+\frac{\theta }{2}+\frac{\theta ^2}{2}\right\}\)

\noindent\(\{3.40016,4.37887\}\)

\noindent\(\left\{\left(
\begin{array}{ccc}
 0 & -1 & 0 \\
 -1 & 0 & 0 \\
 0 & 0 & -1
\end{array}
\right),\left(
\begin{array}{ccc}
 2 & 1 & 2 \\
 1 & 0 & 2 \\
 1 & 0 & 1
\end{array}
\right),-\frac{1}{2}+\frac{\theta }{2}+\frac{\theta ^2}{2},\frac{1}{2}+\frac{\theta }{2}-\frac{\theta ^2}{2}\right\}\)

\noindent\(\{3.40016,6.48864\}\)

\noindent\(\left\{\left(
\begin{array}{ccc}
 1 & -1 & 1 \\
 0 & 1 & 0 \\
 0 & 0 & 1
\end{array}
\right),\left(
\begin{array}{ccc}
 2 & -1 & 4 \\
 1 & -1 & 3 \\
 1 & -1 & 2
\end{array}
\right),-1+\theta ,\frac{1}{2}+\frac{\theta }{2}-\frac{\theta ^2}{2}\right\}\)

\noindent\(\{6.48864,8.05221\}\)

\noindent\(\left\{\left(
\begin{array}{ccc}
 -1 & 0 & 0 \\
 0 & 0 & -1 \\
 0 & -1 & 0
\end{array}
\right),\left(
\begin{array}{ccc}
 -2 & -4 & 1 \\
 -1 & -3 & 1 \\
 -1 & -2 & 1
\end{array}
\right),1+\theta ,2+\theta +\theta ^2\right\}\)

\noindent\(\{6.48864,10.0621\}\)

\noindent\(\left\{\left(
\begin{array}{ccc}
 1 & 0 & 0 \\
 0 & 1 & 1 \\
 0 & 0 & 1
\end{array}
\right),\left(
\begin{array}{ccc}
 -2 & -4 & -3 \\
 -1 & -3 & -2 \\
 -1 & -2 & -1
\end{array}
\right),1+\theta ,1+\theta +\theta ^2\right\}\)

\noindent\(\{10.0621,13.178\}\)

\noindent\(\left\{\left(
\begin{array}{ccc}
 0 & -1 & 0 \\
 -1 & 0 & 0 \\
 0 & 0 & -1
\end{array}
\right),\left(
\begin{array}{ccc}
 4 & 2 & 3 \\
 3 & 1 & 2 \\
 2 & 1 & 1
\end{array}
\right),1+\theta +\theta ^2,1+\theta \right\}\)

\noindent\(\{10.0621,14.8884\}\)

\noindent\(\left\{\left(
\begin{array}{ccc}
 1 & 0 & 0 \\
 0 & 1 & 1 \\
 0 & 0 & 1
\end{array}
\right),\left(
\begin{array}{ccc}
 4 & 2 & 5 \\
 3 & 1 & 3 \\
 2 & 1 & 2
\end{array}
\right),1+\theta +\theta ^2,\theta \right\}\)

\noindent\(\{14.8884,19.3801\}\)

\noindent\(\left\{\left(
\begin{array}{ccc}
 1 & 1 & 0 \\
 0 & 1 & 0 \\
 0 & 0 & 1
\end{array}
\right),\left(
\begin{array}{ccc}
 4 & 6 & 5 \\
 3 & 4 & 3 \\
 2 & 3 & 2
\end{array}
\right),1+\theta ^2,\theta \right\}\)

\noindent\(\{19.3801,20.1874\}\)

\noindent\(\left\{\left(
\begin{array}{ccc}
 1 & 0 & 1 \\
 0 & 1 & 0 \\
 0 & 0 & 1
\end{array}
\right),\left(
\begin{array}{ccc}
 4 & 6 & 9 \\
 3 & 4 & 6 \\
 2 & 3 & 4
\end{array}
\right),\theta ^2,\theta \right\}\)

\noindent\(\{20.1874,43.5621\}\)

\noindent\(\left\{\left(
\begin{array}{ccc}
 1 & 0 & 1 \\
 0 & 1 & 0 \\
 0 & 0 & 1
\end{array}
\right),\left(
\begin{array}{ccc}
 4 & 6 & 13 \\
 3 & 4 & 9 \\
 2 & 3 & 6
\end{array}
\right),-1+\theta ^2,\theta \right\}\) :~period start

~\\
\noindent
$u=\varphi (
\left(
\begin{array}{ccc}
 4 & 6 & 13 \\
 3 & 4 & 9 \\
 2 & 3 & 6
\end{array}
\right)
\left(
\begin{array}{ccc}
 1 & 0 & 1 \\
 0 & 1& 0 \\
 0 & 0 & 1
\end{array}
\right)^{-1}
) = 4+3\theta+2\theta ^2$ is a (fundamental) norm one unit of $\mathbb{Z}[3^{1/3}]$.

}
\end{multicols}

\clearpage

\paragraph{Totally imaginary quartic field $ \mathbb{Q}((-d)^{1/4})$}~

 Let $d$ be a positive quartic free integer. Let $F$ be the totally imaginary quartic field $\mathbb{Q}((-d)^{1/4})$. We consider the Heegner object associated to $w ={}^t\!((-d)^{3/4}, (-d)^{1/2}, (-d)^{1/4}, 1)$, that is, 
 $$\varpi : \mathbb{R}_{>0} \rightarrow \mathfrak{h}^3\  ;\  t \mapsto 
\left[
\begin{array}{cccc}
 -\frac{d^{3/4}}{\sqrt{2}} & \frac{d^{3/4}}{\sqrt{2}} & \frac{d^{3/4} t}{\sqrt{2}} & -\frac{d^{3/4} t}{\sqrt{2}} \\
 \sqrt{d} & 0 & \sqrt{d} t & 0 \\
 -\frac{d^{1/4}}{\sqrt{2}} & -\frac{d^{1/4}}{\sqrt{2}} & \frac{d^{1/4} t}{\sqrt{2}} & \frac{d^{1/4} t}{\sqrt{2}} \\
 0 & 1 & 0 & t
\end{array}
\right] .$$ 
 In the following, we give some examples of the sequence $\{A_k,B_k,[B_k^{-1}w]\}_k$ and $\{s_k , t_k\}_k$ obtained by Algorithm 1 (forward loop) applied to $\varpi(t)$. Here $[v] = (v_1/v_4,v_2/v_4, v_3/v_4)$ for $v={}^t\!(v_1, v_2, v_3, v_4) \in \mathbb{R}^4$. Note that by Corollary \ref{corunitperiod}, $\varpi_k(t)=\varpi_l(\rho t)$ holds for some $\rho \in \mathbb{R}_{>0}$ if and only if $[B_k^{-1}w] = [B_l^{-1}w]$. We write approximate values for $s_k,t_k$. Set $\theta:=(-d)^{1/4}$.
~\\

{\tiny
\noindent\underline{$d=2$}

\noindent\(\left\{preparation; \left(
\begin{array}{cccc}
 1 & 0 & 0 & 0 \\
 0 & 1 & 0 & 0 \\
 0 & 0 & 1 & 0 \\
 0 & 0 & 0 & 1
\end{array}
\right),\theta ^3,\theta ^2,\theta \right\}\) :~period start

\noindent\(\{0.638754,1.56555\}\)

\noindent\(\left\{\left(
\begin{array}{cccc}
 1 & 0 & 0 & -1 \\
 0 & 1 & 0 & 0 \\
 0 & 0 & 1 & 0 \\
 0 & 0 & 0 & 1
\end{array}
\right),\left(
\begin{array}{cccc}
 1 & 0 & 0 & -1 \\
 0 & 1 & 0 & 0 \\
 0 & 0 & 1 & 0 \\
 0 & 0 & 0 & 1
\end{array}
\right),1+\theta ^3,\theta ^2,\theta \right\}\)

\noindent\(\{1.56555,1.76879\}\)

\noindent\(\left\{\left(
\begin{array}{cccc}
 1 & 0 & 0 & 0 \\
 0 & 1 & 1 & 0 \\
 0 & 0 & 1 & 0 \\
 0 & 0 & 0 & 1
\end{array}
\right),\left(
\begin{array}{cccc}
 1 & 0 & 0 & -1 \\
 0 & 1 & 1 & 0 \\
 0 & 0 & 1 & 0 \\
 0 & 0 & 0 & 1
\end{array}
\right),1+\theta ^3,-\theta +\theta ^2,\theta \right\}\)

\noindent\(\{1.76879,1.98329\}\)

\noindent\(\left\{\left(
\begin{array}{cccc}
 1 & 1 & 1 & 0 \\
 0 & 1 & 0 & -1 \\
 0 & 0 & 1 & 1 \\
 0 & 0 & 0 & 1
\end{array}
\right),\left(
\begin{array}{cccc}
 1 & 1 & 1 & -1 \\
 0 & 1 & 1 & 0 \\
 0 & 0 & 1 & 1 \\
 0 & 0 & 0 & 1
\end{array}
\right),1-\theta ^2+\theta ^3,1-\theta +\theta ^2,-1+\theta \right\}\)

\noindent\(\{1.98329,5.09586\}\)

\noindent\(\left\{\left(
\begin{array}{cccc}
 0 & 0 & 1 & 0 \\
 0 & 0 & 0 & 1 \\
 -1 & 0 & 0 & 0 \\
 0 & -1 & 0 & 0
\end{array}
\right),\left(
\begin{array}{cccc}
 -1 & 1 & 1 & 1 \\
 -1 & 0 & 0 & 1 \\
 -1 & -1 & 0 & 0 \\
 0 & -1 & 0 & 0
\end{array}
\right),-\frac{1}{3}-\frac{2 \theta }{3}-\frac{\theta ^2}{3}+\frac{\theta ^3}{3},-\frac{1}{3}+\frac{\theta }{3}+\frac{2 \theta ^2}{3}+\frac{\theta
^3}{3},-\frac{1}{3}+\frac{\theta }{3}-\frac{\theta ^2}{3}+\frac{\theta ^3}{3}\right\}\)

\noindent\(\{2.90629,5.39623\}\)

\noindent\(\left\{\left(
\begin{array}{cccc}
 0 & 1 & 0 & 0 \\
 -1 & 0 & 0 & 0 \\
 0 & 0 & 0 & 1 \\
 0 & 0 & -1 & 0
\end{array}
\right),\left(
\begin{array}{cccc}
 -1 & -1 & -1 & 1 \\
 0 & -1 & -1 & 0 \\
 1 & -1 & 0 & 0 \\
 1 & 0 & 0 & 0
\end{array}
\right),-1+\theta ^2+\theta ^3,1+\theta +\theta ^2,1+\theta \right\}\)

\noindent\(\{2.27056,5.83395\}\)

\noindent\(\left\{\left(
\begin{array}{cccc}
 1 & 1 & -1 & 0 \\
 0 & 1 & 0 & 1 \\
 0 & 0 & 1 & 1 \\
 0 & 0 & 0 & 1
\end{array}
\right),\left(
\begin{array}{cccc}
 -1 & -2 & 0 & -1 \\
 0 & -1 & -1 & -2 \\
 1 & 0 & -1 & -1 \\
 1 & 1 & -1 & 0
\end{array}
\right),-1+\theta ^3,\theta +\theta ^2,\theta \right\}\)

\noindent\(\{5.83395,6.54145\}\)

\noindent\(\left\{\left(
\begin{array}{cccc}
 1 & 0 & 0 & 0 \\
 0 & 1 & 1 & 0 \\
 0 & 0 & 1 & 0 \\
 0 & 0 & 0 & 1
\end{array}
\right),\left(
\begin{array}{cccc}
 -1 & -2 & -2 & -1 \\
 0 & -1 & -2 & -2 \\
 1 & 0 & -1 & -1 \\
 1 & 1 & 0 & 0
\end{array}
\right),-1+\theta ^3,\theta ^2,\theta \right\}\)

\noindent\(\{6.54145,7.39065\}\)

\noindent\(\left\{\left(
\begin{array}{cccc}
 1 & 0 & 0 & -1 \\
 0 & 1 & 0 & 0 \\
 0 & 0 & 1 & 0 \\
 0 & 0 & 0 & 1
\end{array}
\right),\left(
\begin{array}{cccc}
 -1 & -2 & -2 & 0 \\
 0 & -1 & -2 & -2 \\
 1 & 0 & -1 & -2 \\
 1 & 1 & 0 & -1
\end{array}
\right),\theta ^3,\theta ^2,\theta \right\}\) :~period start

\noindent\(\{7.39065,18.1141\}\)

\noindent\(\left\{\left(
\begin{array}{cccc}
 1 & 0 & 0 & -1 \\
 0 & 1 & 0 & 0 \\
 0 & 0 & 1 & 0 \\
 0 & 0 & 0 & 1
\end{array}
\right),\left(
\begin{array}{cccc}
 -1 & -2 & -2 & 1 \\
 0 & -1 & -2 & -2 \\
 1 & 0 & -1 & -3 \\
 1 & 1 & 0 & -2
\end{array}
\right),1+\theta ^3,\theta ^2,\theta \right\}\)

\noindent\(\{18.1141,20.4656\}\)

\noindent\(\left\{\left(
\begin{array}{cccc}
 1 & 0 & 0 & 0 \\
 0 & 1 & 1 & 0 \\
 0 & 0 & 1 & 0 \\
 0 & 0 & 0 & 1
\end{array}
\right),\left(
\begin{array}{cccc}
 -1 & -2 & -4 & 1 \\
 0 & -1 & -3 & -2 \\
 1 & 0 & -1 & -3 \\
 1 & 1 & 1 & -2
\end{array}
\right),1+\theta ^3,-\theta +\theta ^2,\theta \right\}\)

\noindent\(\{20.4656,22.9475\}\)

\noindent\(\left\{\left(
\begin{array}{cccc}
 1 & 1 & 1 & 0 \\
 0 & 1 & 0 & -1 \\
 0 & 0 & 1 & 1 \\
 0 & 0 & 0 & 1
\end{array}
\right),\left(
\begin{array}{cccc}
 -1 & -3 & -5 & -1 \\
 0 & -1 & -3 & -4 \\
 1 & 1 & 0 & -4 \\
 1 & 2 & 2 & -2
\end{array}
\right),1-\theta ^2+\theta ^3,1-\theta +\theta ^2,-1+\theta \right\}\)
}

~\\
From the period we get a (fundamental) norm one unit $u=-1+\theta ^2+\theta ^3$.
\clearpage

\paragraph{$\chi$-component of $\mathbb{Q}(d^{1/4})/\mathbb{Q}(\sqrt{d})$}~

 Let $d>1$ be a positive quartic free integer. Let $F= \mathbb{Q}(d^{1/4})$ and let $F'=\mathbb{Q}(\sqrt{d})$. We consider the $\chi$-Heegner object  for the nontrivial character $\chi$ of $\Gal(F/F')$ associated to $w ={}^t\!(d^{3/4}, d^{1/2}, d^{1/4}, 1)$, that is, 
 $$\varpi : \mathbb{R}_{>0} \rightarrow \mathfrak{h}^3\  ;\  t \mapsto 
\left[
\begin{array}{cccc}
 -d^{3/4} t & 0 & -d^{3/4} & d^{3/4} t^2 \\
 0 & -\sqrt{d} t & \sqrt{d} & \sqrt{d} t^2 \\
 d^{1/4} t & 0 & -d^{1/4} & d^{1/4} t^2 \\
 0 & t & 1 & t^2
\end{array}
\right] .$$ 
 In the following, we give some examples of the sequence $\{A_k,B_k,[B_k^{-1}w]\}_k$ and $\{s_k , t_k\}_k$ obtained by Algorithm 1 (forward loop) applied to $\varpi(t)$. Here $[v] = (v_1/v_4,v_2/v_4,v_3/v_4)$ for $v={}^t\!(v_1, v_2, v_3,v_4) \in \mathbb{R}^4$. Note that by Corollary \ref{corunitperiod} and Proposition \ref{chiunitperiod}, $\varpi_k(t)=\varpi_l(\rho t)$ holds for some $\rho \in \mathbb{R}_{>0}$ if and only if $[B_k^{-1}w] = [B_l^{-1}w]$ and $N_{F/F'}(\varphi (B_{l}B_{k}^{-1}))=1$. We write approximate values for $s_k,t_k$. Set $\theta := d^{1/4}$ 
~\\

{\tiny
\noindent\underline{$d=2$}

\noindent\(\left\{preparation; \left(
\begin{array}{cccc}
 1 & 0 & 0 & 0 \\
 0 & 1 & 0 & 0 \\
 0 & 0 & 1 & 0 \\
 0 & 0 & 0 & 1
\end{array}
\right),\theta ^3,\theta ^2,\theta \right\}\) :~period start

\noindent\(\{0.858498,1.16483\}\)

\noindent\(\left\{\left(
\begin{array}{cccc}
 1 & 0 & 0 & 0 \\
 0 & 1 & 0 & 1 \\
 0 & 0 & 1 & 0 \\
 0 & 0 & 0 & 1
\end{array}
\right),\left(
\begin{array}{cccc}
 1 & 0 & 0 & 0 \\
 0 & 1 & 0 & 1 \\
 0 & 0 & 1 & 0 \\
 0 & 0 & 0 & 1
\end{array}
\right),\theta ^3,-1+\theta ^2,\theta \right\}\)

\noindent\(\{1.16483,1.25962\}\)

\noindent\(\left\{\left(
\begin{array}{cccc}
 1 & 0 & 0 & 1 \\
 0 & 1 & 0 & 0 \\
 0 & 0 & 1 & 0 \\
 0 & 0 & 0 & 1
\end{array}
\right),\left(
\begin{array}{cccc}
 1 & 0 & 0 & 1 \\
 0 & 1 & 0 & 1 \\
 0 & 0 & 1 & 0 \\
 0 & 0 & 0 & 1
\end{array}
\right),-1+\theta ^3,-1+\theta ^2,\theta \right\}\)

\noindent\(\{1.25962,1.40348\}\)

\noindent\(\left\{\left(
\begin{array}{cccc}
 1 & 0 & 0 & 0 \\
 0 & 1 & 0 & 0 \\
 0 & 0 & 1 & 1 \\
 0 & 0 & 0 & 1
\end{array}
\right),\left(
\begin{array}{cccc}
 1 & 0 & 0 & 1 \\
 0 & 1 & 0 & 1 \\
 0 & 0 & 1 & 1 \\
 0 & 0 & 0 & 1
\end{array}
\right),-1+\theta ^3,-1+\theta ^2,-1+\theta \right\}\)

\noindent\(\{1.40348,1.64805\}\)

\noindent\(\left\{\left(
\begin{array}{cccc}
 1 & 0 & 0 & 0 \\
 0 & 1 & 1 & 0 \\
 0 & 0 & 1 & 0 \\
 0 & 0 & 0 & 1
\end{array}
\right),\left(
\begin{array}{cccc}
 1 & 0 & 0 & 1 \\
 0 & 1 & 1 & 1 \\
 0 & 0 & 1 & 1 \\
 0 & 0 & 0 & 1
\end{array}
\right),-1+\theta ^3,-\theta +\theta ^2,-1+\theta \right\}\)

\noindent\(\{1.64805,1.89353\}\)

\noindent\(\left\{\left(
\begin{array}{cccc}
 1 & 1 & 1 & 0 \\
 0 & 1 & 0 & 0 \\
 0 & 0 & 1 & 0 \\
 0 & 0 & 0 & 1
\end{array}
\right),\left(
\begin{array}{cccc}
 1 & 1 & 1 & 1 \\
 0 & 1 & 1 & 1 \\
 0 & 0 & 1 & 1 \\
 0 & 0 & 0 & 1
\end{array}
\right),-\theta ^2+\theta ^3,-\theta +\theta ^2,-1+\theta \right\}\)

\noindent\(\{1.89353,1.953\}\)

\noindent\(\left\{\left(
\begin{array}{cccc}
 1 & 0 & 0 & 0 \\
 0 & -1 & 0 & 0 \\
 0 & 0 & 0 & -1 \\
 0 & 0 & -1 & 0
\end{array}
\right),\left(
\begin{array}{cccc}
 1 & -1 & -1 & -1 \\
 0 & -1 & -1 & -1 \\
 0 & 0 & -1 & -1 \\
 0 & 0 & -1 & 0
\end{array}
\right),-\theta ^2,\theta ,1+\theta +\theta ^2+\theta ^3\right\}\)

\noindent\(\{1.89353,2.21428\}\)

\noindent\(\left\{\left(
\begin{array}{cccc}
 1 & 0 & 0 & 1 \\
 0 & 1 & 0 & 0 \\
 0 & 0 & 1 & 0 \\
 0 & 0 & 0 & 1
\end{array}
\right),\left(
\begin{array}{cccc}
 1 & -1 & -1 & 0 \\
 0 & -1 & -1 & -1 \\
 0 & 0 & -1 & -1 \\
 0 & 0 & -1 & 0
\end{array}
\right),-1-\theta ^2,\theta ,1+\theta +\theta ^2+\theta ^3\right\}\)

\noindent\(\{2.21428,2.28296\}\)

\noindent\(\left\{\left(
\begin{array}{cccc}
 1 & 0 & 0 & 0 \\
 0 & 0 & -1 & 0 \\
 0 & -1 & 0 & 0 \\
 0 & 0 & 0 & -1
\end{array}
\right),\left(
\begin{array}{cccc}
 1 & 1 & 1 & 0 \\
 0 & 1 & 1 & 1 \\
 0 & 1 & 0 & 1 \\
 0 & 1 & 0 & 0
\end{array}
\right),1+\theta ^2,1+\theta +\theta ^2+\theta ^3,\theta \right\}\)

\noindent\(\{2.21428,2.43997\}\)

\noindent\(\left\{\left(
\begin{array}{cccc}
 0 & -1 & 0 & 0 \\
 1 & 0 & 0 & 0 \\
 0 & 0 & -1 & 0 \\
 0 & 0 & 0 & -1
\end{array}
\right),\left(
\begin{array}{cccc}
 1 & -1 & -1 & 0 \\
 1 & 0 & -1 & -1 \\
 1 & 0 & 0 & -1 \\
 1 & 0 & 0 & 0
\end{array}
\right),-1-\theta -\theta ^2-\theta ^3,1+\theta ^2,\theta \right\}\)

\noindent\(\{2.21428,2.51163\}\)

\noindent\(\left\{\left(
\begin{array}{cccc}
 1 & -1 & 0 & 0 \\
 0 & 1 & 0 & 0 \\
 0 & 0 & 1 & 0 \\
 0 & 0 & 0 & 1
\end{array}
\right),\left(
\begin{array}{cccc}
 1 & -2 & -1 & 0 \\
 1 & -1 & -1 & -1 \\
 1 & -1 & 0 & -1 \\
 1 & -1 & 0 & 0
\end{array}
\right),-\theta -\theta ^3,1+\theta ^2,\theta \right\}\)

\noindent\(\{2.51163,4.60674\}\)

\noindent\(\left\{\left(
\begin{array}{cccc}
 1 & -1 & 0 & 0 \\
 0 & 1 & 0 & 0 \\
 0 & 0 & 1 & 0 \\
 0 & 0 & 0 & 1
\end{array}
\right),\left(
\begin{array}{cccc}
 1 & -3 & -1 & 0 \\
 1 & -2 & -1 & -1 \\
 1 & -2 & 0 & -1 \\
 1 & -2 & 0 & 0
\end{array}
\right),1-\theta +\theta ^2-\theta ^3,1+\theta ^2,\theta \right\}\)

\noindent\(\{4.60674,5.22538\}\)

\noindent\(\left\{\left(
\begin{array}{cccc}
 1 & 0 & 0 & 0 \\
 0 & 1 & 0 & 1 \\
 0 & 0 & 1 & 0 \\
 0 & 0 & 0 & 1
\end{array}
\right),\left(
\begin{array}{cccc}
 1 & -3 & -1 & -3 \\
 1 & -2 & -1 & -3 \\
 1 & -2 & 0 & -3 \\
 1 & -2 & 0 & -2
\end{array}
\right),1-\theta +\theta ^2-\theta ^3,\theta ^2,\theta \right\}\)

\noindent\(\{5.22538,6.03036\}\)

\noindent\(\left\{\left(
\begin{array}{cccc}
 1 & 0 & 0 & 0 \\
 0 & 1 & 1 & -1 \\
 0 & 0 & 1 & 0 \\
 0 & 0 & 0 & 1
\end{array}
\right),\left(
\begin{array}{cccc}
 1 & -3 & -4 & 0 \\
 1 & -2 & -3 & -1 \\
 1 & -2 & -2 & -1 \\
 1 & -2 & -2 & 0
\end{array}
\right),1-\theta +\theta ^2-\theta ^3,1-\theta +\theta ^2,\theta \right\}\)

\noindent\(\{6.03036,6.18809\}\)

\noindent\(\left\{\left(
\begin{array}{cccc}
 0 & 1 & 0 & 0 \\
 -1 & 0 & 0 & 0 \\
 0 & 0 & -1 & 0 \\
 0 & 0 & 0 & -1
\end{array}
\right),\left(
\begin{array}{cccc}
 3 & 1 & 4 & 0 \\
 2 & 1 & 3 & 1 \\
 2 & 1 & 2 & 1 \\
 2 & 1 & 2 & 0
\end{array}
\right),1-\theta +\theta ^2,-1+\theta -\theta ^2+\theta ^3,\theta \right\}\)

\noindent\(\{6.03036,6.87183\}\)

\noindent\(\left\{\left(
\begin{array}{cccc}
 1 & 0 & 0 & 0 \\
 0 & 0 & -1 & 0 \\
 0 & -1 & 0 & 0 \\
 0 & 0 & 0 & -1
\end{array}
\right),\left(
\begin{array}{cccc}
 3 & -4 & -1 & 0 \\
 2 & -3 & -1 & -1 \\
 2 & -2 & -1 & -1 \\
 2 & -2 & -1 & 0
\end{array}
\right),-1+\theta -\theta ^2,\theta ,-1+\theta -\theta ^2+\theta ^3\right\}\)

\noindent\(\{6.1105,7.40949\}\)

\noindent\(\left\{\left(
\begin{array}{cccc}
 1 & 0 & 0 & 0 \\
 0 & 1 & 0 & 1 \\
 0 & 0 & 1 & 0 \\
 0 & 0 & 0 & 1
\end{array}
\right),\left(
\begin{array}{cccc}
 3 & -4 & -1 & -4 \\
 2 & -3 & -1 & -4 \\
 2 & -2 & -1 & -3 \\
 2 & -2 & -1 & -2
\end{array}
\right),-1+\theta -\theta ^2,-1+\theta ,-1+\theta -\theta ^2+\theta ^3\right\}\)

\noindent\(\{7.40949,7.9813\}\)

\noindent\(\left\{\left(
\begin{array}{cccc}
 1 & 0 & 0 & 0 \\
 0 & -1 & 0 & 0 \\
 0 & 0 & 0 & -1 \\
 0 & 0 & -1 & 0
\end{array}
\right),\left(
\begin{array}{cccc}
 3 & 4 & 4 & 1 \\
 2 & 3 & 4 & 1 \\
 2 & 2 & 3 & 1 \\
 2 & 2 & 2 & 1
\end{array}
\right),1+\theta ^3,-1+\theta ^2,1+\theta \right\}\)

\noindent\(\{7.02068,8.24409\}\)

\noindent\(\left\{\left(
\begin{array}{cccc}
 1 & 0 & 0 & 0 \\
 0 & 1 & 0 & 0 \\
 0 & 0 & 1 & 1 \\
 0 & 0 & 0 & 1
\end{array}
\right),\left(
\begin{array}{cccc}
 3 & 4 & 4 & 5 \\
 2 & 3 & 4 & 5 \\
 2 & 2 & 3 & 4 \\
 2 & 2 & 2 & 3
\end{array}
\right),1+\theta ^3,-1+\theta ^2,\theta \right\}\)

\noindent\(\{8.24409,9.18562\}\)

\noindent\(\left\{\left(
\begin{array}{cccc}
 1 & 0 & 0 & 1 \\
 0 & 1 & 0 & 0 \\
 0 & 0 & 1 & 0 \\
 0 & 0 & 0 & 1
\end{array}
\right),\left(
\begin{array}{cccc}
 3 & 4 & 4 & 8 \\
 2 & 3 & 4 & 7 \\
 2 & 2 & 3 & 6 \\
 2 & 2 & 2 & 5
\end{array}
\right),\theta ^3,-1+\theta ^2,\theta \right\}\)

\noindent\(\{9.18562,9.93319\}\)

\noindent\(\left\{\left(
\begin{array}{cccc}
 1 & 0 & 0 & 0 \\
 0 & 1 & 0 & -1 \\
 0 & 0 & 1 & 0 \\
 0 & 0 & 0 & 1
\end{array}
\right),\left(
\begin{array}{cccc}
 3 & 4 & 4 & 4 \\
 2 & 3 & 4 & 4 \\
 2 & 2 & 3 & 4 \\
 2 & 2 & 2 & 3
\end{array}
\right),\theta ^3,\theta ^2,\theta \right\}\) :~period start

\noindent\(\{9.93319,13.4775\}\)

\noindent\(\left\{\left(
\begin{array}{cccc}
 1 & 0 & 0 & 0 \\
 0 & 1 & 0 & 1 \\
 0 & 0 & 1 & 0 \\
 0 & 0 & 0 & 1
\end{array}
\right),\left(
\begin{array}{cccc}
 3 & 4 & 4 & 8 \\
 2 & 3 & 4 & 7 \\
 2 & 2 & 3 & 6 \\
 2 & 2 & 2 & 5
\end{array}
\right),\theta ^3,-1+\theta ^2,\theta \right\}\)

\noindent\(\{13.4775,14.5744\}\)

\noindent\(\left\{\left(
\begin{array}{cccc}
 1 & 0 & 0 & 1 \\
 0 & 1 & 0 & 0 \\
 0 & 0 & 1 & 0 \\
 0 & 0 & 0 & 1
\end{array}
\right),\left(
\begin{array}{cccc}
 3 & 4 & 4 & 11 \\
 2 & 3 & 4 & 9 \\
 2 & 2 & 3 & 8 \\
 2 & 2 & 2 & 7
\end{array}
\right),-1+\theta ^3,-1+\theta ^2,\theta \right\}\)

}
~\\
From the period we get a norm one $\chi$-unit $u=3+2\theta+2\theta ^2 +2\theta ^3$.

\subsection{$\{\infty,p\}$-continued fractions}

Let $d$ be a positive square free integer. Let $F=\mathbb{Q}(\sqrt{-d}) \subset \mathbb{C} \simeq \overline{\mathbb{Q}}_p$ and $p$ a prime number which splits completely in $F/\mathbb{Q}$. Here we fix an isomorphism of fields $\mathbb{C} \simeq \overline{\mathbb{Q}}_p$. We use the same notations as in Section \ref{sectp-app}. \\

\noindent\underline{$d=1, p=5$}\\
Suppose $\sqrt{-1} \equiv 2 \mod 5$ under the fixed isomorphism. Then the $\{\infty,p\}$-continued fraction expansion of $\sqrt{-1}$ is as follows.
\begin{eqnarray*}
\sqrt{-1} &=& [\overline{+1;7;0,4,1};] \\
			 &=& 7- \cfrac{25}{4- \cfrac{1}{1- \cfrac{1}{7- \cfrac{25}{4-\cfrac{1}{ \ddots}}}}}
\end{eqnarray*}
where $[-;\overline{\sim}; ]:=[-;\sim ; \sim ; \cdots ]$, and we simplify the continued fraction form by deleting the zero-terms as $- \cfrac{1}{0-\cfrac{1}{a}}=a$. 
Then, $B_1^{-1}= 
\left(
\begin{array}{cc}
  4/5&-3/5    \\
  3/5&4/5
\end{array}
\right)$ gives a fundamental norm one $p$-unit $\epsilon = (3\sqrt{-1} +4)/5$, and $(x,y,k)=(4,3,1)$ gives the solution to $x^2+y^2=5^{2k}$. Furthermore, the identity $(3\sqrt{-1}+4)/(-\sqrt{-1}+2)^2=\sqrt{-1}$ implies ${\rm ord} [(2-\sqrt{-1})] =1$. See Theorem \ref{p-lag2}.\\

\noindent\underline{$d=5, p=3$}\\
Suppose $\sqrt{-5} \equiv 1 \mod 3$ under the fixed isomorphism. Then the $\{\infty,p\}$-continued fraction expansion of $\sqrt{-1}$ is as follows.
\begin{eqnarray*}
\sqrt{-5} &=& [\overline{+1;7;-1,-2,0};] \\
			 &=& -2- \cfrac{9}{-1- \cfrac{1}{-4- \cfrac{9}{-1- \cfrac{1}{-4-\cfrac{9}{ \ddots.}}}}}
\end{eqnarray*}
Here we simplify the continued fraction form by deleting the zero-terms as $- \cfrac{1}{0-\cfrac{1}{a}}=a$.
Then, $B_1^{-1}= 
\left(
\begin{array}{cc}
  -2/3&5/3    \\
  -1/3&-2/3
\end{array}
\right)$ gives a fundamental norm one $p$-unit $\epsilon = (-\sqrt{-5} -2)/3$, and $(x,y,k)=(-2,-1,1)$ gives the solution to $x^2+5y^2=3^{2k}$. Furthermore, ${\rm ord} [(3,1-\sqrt{-5})] =2$ since $\epsilon$ is square free in $\mathbb{Z}[\sqrt{-5}]$. See Theorem \ref{p-lag2}.\\

\noindent\underline{$d=14, p=3$}\\
Suppose $\sqrt{-14} \equiv 2 \mod 3$ under the fixed isomorphism. Then the $\{\infty,p\}$-continued fraction expansion of $\sqrt{-1}$ is as follows.
\begin{eqnarray*}
\sqrt{-14} &=& [\overline{+1;2;0,1,0;+1;1;0,2,0};] \\
			 &=& 2- \cfrac{9}{2- \cfrac{9}{4- \cfrac{9}{2- \cfrac{9}{4-\cfrac{9}{ \ddots.}}}}}
\end{eqnarray*}
Here we simplify the continued fraction form by deleting the zero-terms as $- \cfrac{1}{0-\cfrac{1}{a}}=a$.
Then, $B_2^{-1}= 
\left(
\begin{array}{cc}
  -5/9&28/9    \\
  -2/9&-5/9
\end{array}
\right)$ gives a fundamental norm one $p$-unit $\epsilon = (-2\sqrt{-14} -5)/9$, and $(x,y,k)=(-5,-2,2)$ gives the solution to $x^2+14y^2=3^{2k}$. Furthermore, ${\rm ord} [(3,2-\sqrt{-14})] =4$ since $\epsilon$ is square free in $\mathbb{Z}[\sqrt{-14}]$. See Theorem \ref{p-lag2}.

\clearpage
\subsection*{Acknowledgements}
I would like to express my deepest gratitude to my advisor Professor Takeshi Tsuji for the constant encouragement, valuable suggestions and helpful comments.


 
\end{document}